\documentclass[11pt,letterpaper]{amsart}
\usepackage{color}
\usepackage{stmaryrd}
\usepackage{cases}
\usepackage{amssymb}
\usepackage{amsmath}
\usepackage{amsfonts}
\usepackage{graphicx}
\usepackage{amsmath,amstext,amsbsy,amssymb, color}
\usepackage[colorlinks=true, linkcolor=blue, citecolor=magenta, urlcolor=blue, backref=page]{hyperref}

\usepackage{latexsym}
\usepackage{amsmath}
\usepackage{amssymb}
\usepackage{bm}
\usepackage{graphicx}
\usepackage{wrapfig}
\usepackage{fancybox}

\newtheorem{theorem}{Theorem}[section]
\newtheorem{lemma}[theorem]{Lemma}
\newtheorem{proposition}[theorem]{Proposition}
\newtheorem{corollary}[theorem]{Corollary}
\theoremstyle{definition}

\theoremstyle{remark}
\newtheorem{remark}[theorem]{Remark}

\numberwithin{equation}{section} \errorcontextlines=0

\newcommand{\per}{\mathrm{per}}
\newcommand{\Imm}{\mathrm{Imm}}

\newcommand{\ot}{\otimes}

\newcommand{\si}{\sigma}
\newcommand{\GL}{\mathrm{GL}}
\newcommand{\gl}{\mathfrak{gl}}

\newcommand{\Mat}{\mathrm{Mat}}

\newcommand{\Hc}{\mathcal{H}}
\newcommand{\Ec}{\mathcal{E}}
\newcommand{\Tc}{\mathcal{T}}
\newcommand{\bal}{\begin{aligned}}
\newcommand{\eal}{\end{aligned}}
\newcommand{\beq}{\begin{equation}}
\newcommand{\eeq}{\end{equation}}
\newcommand{\ben}{\begin{equation*}}
\newcommand{\een}{\end{equation*}}
\newcommand{\Sym}{\mathfrak S}
\newcommand{\CC}{\mathbb{C}}

\begin{document}
	
	\title[Quantum Littlewood correspondences]
	{Quantum Littlewood correspondences}
	
	\author{Naihuan Jing}
	\address{Department of Mathematics, North Carolina State University, Raleigh, NC 27695, USA}
	\email{jing@ncsu.edu}
	
		\author{Yinlong Liu}
	\address{Department of Mathematics, Shanghai University, Shanghai 200444, China}
	\email{yinlongliu@shu.edu.cn}
	
	\author{Jian Zhang}
	\address{School of Mathematics and Statistics,
		Central China Normal University, Wuhan, Hubei 430079, China}
	\email{jzhang@ccnu.edu.cn}
	
	\thanks{{\scriptsize
			\hskip -0.6 true cm MSC (2020): Primary: 17B37 Secondary: 20G05, 17B35, 17B66, 05E10
	}}
	

\begin{abstract} In the 1940s Littlewood formulated three fundamental correspondences for the immanants and Schur symmetric functions on the general linear group, which establish deep connections between
representation theory of the symmetric group and the general linear group parallel to the Schur-Weyl duality.
In this paper, we introduce the notion of quantum immanants in the quantum coordinate algebra using primitive idempotents of the Hecke algebra.
By employing $R$-matrix techniques, we establish the quantum analog of Littlewood correspondences between quantum immanants and Schur functions
for the quantum coordinate algebra.
In the setting of the Schur-Weyl-Jimbo duality, we construct an exact correspondence between the Gelfand-Tsetlin bases of the irreducible
representations of the
quantum enveloping algebra $U_q(\mathfrak{gl}(n))$ and Young's orthonormal basis of an irreducible representation of the Hecke algebra $\mathcal H_m$. This isomorphism leads to
our trace formula for the quantum immanants, which settled the generalization problem
 of $q$-analog of Kostant's formular for $\lambda$-immanants. As applications, we also derive general $q$-Littlewood-Merris-Watkins identities and $q$-Goulden-Jackson identities as special cases of the quantum Littlewood correspondence III.
\end{abstract}
	\maketitle
\section{Introduction}

The immanants appeared first in the work of Schur \cite{S} and Littlewood \cite{Li, LR} coined the name in his well-known work on the representation theory of symmetric groups and general linear groups. Let $\chi^{\lambda}$  be the complex irreducible character of the symmetric group $\Sym_n$ indexed by partition $\lambda$ of $n$. The $\lambda$-immanant of the $n \times n$ matrix $A=(a_{ij})$ is defined by
\ben
\Imm_{\chi^{\lambda}}(A)=\sum_{\si\in \Sym_n}\chi^{\lambda}(\si)\prod_{i=1}^{n}a_{\si_i,i},
\een
which includes the determinant (sign character) and permanent (trivial character) as special cases.
Schur \cite{S} showed that $\Imm_{\chi^{\lambda}}(A)\geq \chi^{\lambda}(1)\det(A)  $ for any positive semidefinite Hermitian matrix $A$ and any partition $\lambda$ of $n$.

Littlewood \cite{Li} formulated three correspondences between Schur polynomials and immanants for the symmetric group $\mathfrak{S}_n$ and the general linear group $\mathrm{GL}_n$.  Littlewood's first and second correspondences
say that any relation between Schur polynomials remains the same by replacing the Schur functions $S_{\lambda}$ with the associated immanants
$\Imm_{\lambda}$ of the principal complementary or principal minors of the general matrix $X$. By the isomorphism of the algebra of class functions of the symmetric groups and the algebra of symmetric functions, Schur polynomials correspond to the irreducible representations of symmetric groups. The Schur polynomials can also be identified with the irreducible representations of general linear groups.
Littlewood's first two correspondences reveal deep relations among
immanant identities from the representation theory of the symmetric and general linear groups. In other words, the immanant identities provide exact numerical
relations for the representation theory of the symmetric group and the general linear group in a spirit similar to the Schur-Weyl duality.

Littlewood's third correspondence says that the Schur function $S_{\lambda}$ can be explicitly expressed by the normalized immanants:
\begin{equation}\label{e:Littlewood3-classical}
S_{\lambda}(\omega_1, \ldots, \omega_n)=\sum_{I} \frac{\Imm_{\chi^{\lambda}}(X_I)}{m(I)}
\end{equation}
summed over all nondecreasing ordered multisubsets $I=(i_1\leq\ldots\leq i_m)$ of $[n]$ and $m(I)=\prod_{i\in I} m_i!$, where
$m_i=m_i(I)$. Here $\omega_i$ are the characteristic roots of the matrix $X$.

 Littlewood correspondence I and II have some well-known consequences. In particular, the generalized Littlewood--Merris-Watkins identities \cite{KS,Li,MW} can be derived as a special case of
 Littlewood correspondences. The Goulden-Jackson identities \cite{GJ} are also a direct corollary from Littlewood correspondence III. It is important to note that
 the Goulden-Jackson identities  \cite{GJ} also give new interpretations of the Littlewood correspondence, MacMahon theorem \cite{Mac,VJ}, and Young idempotents.

  In another connection with representation theory,
 Haiman \cite{H} established the connection between Kazhdan-Lusztig theory and the character  of Hecke algebras.
 Kostant \cite{Ko} interpreted the immanants using $0$-weight spaces for all representations of $\mathrm{SL}(n,\CC)$, which
 generalized some early work of Schur \cite{S}. Kostant's formulation also provides a trace formula for the $\lambda$-immanants without repeating indices.  It has been an open problem whether Kostant's
 formula can be generalized to all weight spaces.


In the late 1990s Okounkov \cite{Ok} introduced the concept of quantum $\lambda$-immanants in the enveloping algebra $U(\mathfrak{gl}(n))$ associated with the general linear Lie algebra. He showed that
the $\lambda$-immanants of $U(\mathfrak{gl}(n))$ correspond to a special basis in the center of the enveloping algebra $U(\mathfrak{gl}(n))$. Okounkov's quantum immanants have also been studied from combinatorial viewpoints \cite{BT}. Recently
M. Liu, A. Molev, and one of us \cite{JLM} have generalized these structures to the quantum enveloping algebra $U_q(\mathfrak{gl}(n))$ and
introduced the $q$-immanants of $U_q(\mathfrak{gl}(n))$ as certain special elements in the center $Z(U_q(\mathfrak{gl}(n)))$ indexed by partitions of $n$, which
deform Okounkov's $\lambda$-immanants. The $q$-immanants determined the higher Capelli identities and their Harish-Chandra image
are also factorial Schur polynomials.

In this paper, we will study the quantum coordinate algebra
to deform Littlewood's theory. Early attempt to generalize immanants to the quantum coordinate algebra
 was made by Konvalinka and Skandera \cite{KS}, who verified some special cases of quantum Littlewood-Merris-Watkins identities.
However, the quantum version of general Littlewood-Merris-Watkins identities and Goulden-Jackson identities remain open, since there were
no conjectures on how to
formulate the notion of quantum immanants on the general quantum group setting, that is, on the quantum general linear group $\GL_q(n)$ and more generally on
 the quantum coordinate algebra $A_q(\Mat_n)$ so that Littlewood's theory of immanants can
be completely quantized in connection with the Hecke algebra  $\mathcal H_m$ of type $A$.

One of the main results of this paper is to introduce an appropriate notion of quantum immanants in the quantum coordinate algebra $A_q(\Mat_n)$  and
the Hecke algebra $\mathcal H_m$ and formulate the quantum Littlewood correspondences.
We also show that the general  Littlewood-Merris-Watkins identities and general Goulden-Jackson identities hold on the quantum coordinate algebra.

To formulate the right notion of quantum immanants on
 $A_q(\Mat_n(X))$, we realize that one needs to extend the meaning of $\Imm_{\chi}(X_I^J)$
 to more general arguments than simply a submatrix $X_I^J$ and also beyond irreducible character $\chi$ of $\mathcal H_m$. We solve
 this problem by using the Schur-Weyl-Jimbo duality on the tensor space $(\mathbb C^n)^{\otimes m}$
 to define the quantum immanants associated
 to any character $\chi$ of $\mathcal H_m$ by
 \begin{equation}
 \Imm_{\chi}(X_I^J)=\langle i_{1},\ldots ,i_{m}\mid \chi X_1\cdots X_{m} \mid j_{1},\ldots,j_{m}\rangle
 \end{equation}
 where $\langle\ , \ \rangle$ is the canonical inner product on $(\mathbb C^n)^{\otimes m}$. Here the character $\chi$ is
 defined as an element of the Hecke algebra $\mathcal H_m$ (see Section \ref{S:HeckeChar}).

We also introduce the Bethe subalgebra of the quantum coordinate algebra and show that it is isomorphic to the ring of classical symmetric functions.
We find the exact correspondence between the Gelfand-Tsetlin basis for the quantum general linear algebra and the Young orthonormal basis of tableaux in
the context of the Schur-Weyl-Jimbo duality.  Moreover, we derive the general trace identity for the quantum immanants (Theorem \ref{q-Kostant-thm}):
\begin{equation}\label{e:corresp}
\frac{\Imm_{\chi_q^{\lambda}}(X_I)}{m_{q^2}(I)}=tr(\mathcal P_{\mu}\ot 1)\circ\Delta\circ \mathcal P_{\mu}|_{U^{\lambda}},
\end{equation}
where $\mathcal P_{\mu}$ is the projection onto the weight $\mu$-subspace of $(\mathbb C^n)^{\otimes m}|_{\mu}$ and $U^{\lambda}$ is the irreducible
$U_q(\mathfrak{gl}_n)$-module inside $(\mathbb C^n)^{\otimes m}$. The identity \eqref{e:corresp} provides a representation theoretical interpretation of the quantum immanants and contains
the Kostant trace identity as special cases when $\mu=(1^n)$ under $q\to 1$.

Let's describe how we obtain the quantum version of the Littlewood correspondence III. Let $X=(x_{ij})_{n\times n}$ be the generator matrix
of the quantum coordinate algebra $A_q(\Mat_n)$. The quantum elementary symmetric functions $\alpha_k$ can be defined by
\begin{equation}
\alpha_k=tr\mathcal E^{(1^k)} X_1X_2\cdots X_k=\sum_{|I|=k}{\det}_q(X), 
\end{equation}
where $I$ runs through all multisubsets $I$ of cardinality $k$. The functions $\alpha_k$ correspond to the noncommutative elementary
symmetric functions \cite{Mo} on the
general linear Lie algebra, moreover, it turns out that these $\alpha_k$ are commutative to each other \cite{DL} and
generate the Bethe subalgebra
of $A_q(\Mat(n))$. We will
also show this fact by using the RTT relations (Theorem \ref{t:commutativity-alpha}) and prove that
they correspond to the central elements in the center of $U_q(\mathfrak{gl}(n))$.
We then obtain the
$q$-Goulden-Jackson identity (Theorem \ref{quantum-JT}):
\begin{equation}
\sum_{\mu}K_{\lambda^{T},\mu}^{-1}\alpha_{\mu_1}\cdots \alpha_{\mu_{\lambda_1}}=tr\mathcal E^{\lambda} X_1X_2\cdots X_k=\sum_{|I|=k}\frac{\Imm_{\chi_q^{\lambda}}(X_I)}{m_{q^2}(I)}
\end{equation}
where the first sum is over all partitions $\mu$ and $K_{\lambda, \mu}$ are the Kostka numbers, the second
sum runs over all non-decreasing ordered cardinality $r$ multisets $I$ of $[n]$.

Let $\lambda \vdash r\leq n$ and $\omega_1,\dots,\omega_n$ be the roots of the $q$-characteristic polynomial $\text{char}_q(X,t)=\sum_{k=0}^{n}(-1)^k \alpha_{k}t^{n-k}$ over the algebraic closure field of fraction $\mathbb{C}(q, \alpha_1,\ldots,\alpha_n)$. The
quantum version of
Littlewood coreespondence III (Theorem \ref{t:Litt3}) says that the quantum immanants can be expressed
as
\begin{equation}
\mathbb{S}_{\lambda}(\omega_1,\dots,\omega_n)=\sum_{I}\frac{\Imm_{\chi_q^{\lambda}}(X_I)}{m_{q^2}(I)},
\end{equation}
where the sum is over all non-decreasing ordered multisets $I$ of $[n]$  satisfying $|I|=r$, which is an exact
$q$-deformation of the classical Littlewood correspondence III \eqref{e:Littlewood3-classical}.

The paper is organized as follows.
In Section 2, we recall the properties of quantum coordinate algebra and prepare the necessary materials for the representation theory of the Hecke algebra. In Section 3, we introduce a new definition of quantum immanants, which is different from Konvalinka and Skandera's, and obtain the quantization of the Kostant Theorem. In Section 4, we study the quantum analog of Littlewood correspondence I and II. And as special cases, we show that the general Littlewood-Merris-Watkins
identities. In Section 5, we prove the $q$-Goulden-Jackson identities. Meanwhile, some classical identities for symmetric functions are generalized to the quantum coordinate algebra. Finally, we derive the quantum analog of the Littlewood correspondence III.

\section{The  quantum coordinate algebra and the Hecke algebra}
\subsection{Quantum algebra $A_{q}(\Mat_{n})$} Throughout the paper, we assume that $q$ is not a root of unity.
Let $R$ be the matrix in $\mathrm{End}(\mathbb{C}^n\ot\mathbb{C}^n)$:
\begin{equation}\label{R}
  R=q\sum_{i=1}^{n}e_{ii}\ot e_{ii}+\sum_{ i\neq j} e_{ii}\ot e_{jj}+(q-q^{-1})\sum_{ i<j} e_{ij}\ot e_{ji},
\end{equation}
where $e_{ij}$ be the standard basis of $\mathrm{End}(\mathbb{C}^n)$. The  $R$ satisfies the well-known Yang-Baxter equation (YBE):
\begin{equation}\label{YBeq}
  R_{12} R_{13} R_{23}=R_{23} R_{13} R_{12},
\end{equation}
where  $R_{i j} \in \mathrm{End}(\mathbb{C}^{n} \otimes \mathbb{C}^{n} \otimes \mathbb{C}^{n})$  acts on the  $i$th and  $j$th copies of  $\mathbb{C}^{n}$  as  $R$  does on  $\mathbb{C}^{n} \otimes \mathbb{C}^{n}$.

We define two $R$-matrices $R^{\pm}$ associated with $R$ by $R^+=PRP$, $R^-=R^{-1}$. Explicitly
\begin{align}
  R^+&=q\sum_{i=1}^{n}e_{ii}\ot e_{ii}+\sum_{ i\neq j} e_{ii}\ot e_{jj}+(q-q^{-1})\sum_{ i>j} e_{ij}\ot e_{ji},\\
  R^-&=q^{-1}\sum_{i=1}^{n}e_{ii}\ot e_{ii}+\sum_{ i\neq j} e_{ii}\ot e_{jj}-(q-q^{-1})\sum_{i<j} e_{ij}\ot e_{ji}.
\end{align}

The {\it  quantum coordinate algebra} $A_{q}(\Mat_{n})$ is unital associative algebra generated by  $x_{i j}, 1 \leq i, j \leq n$  subject to the quadratic relations defined by the matrix equation
\begin{equation}\label{RT}
  R X_{1} X_{2}=X_{2} X_{1} R
\end{equation}
in  $A_{q}(\Mat_{n}) \otimes \mathrm{End}(\mathbb{C}^{n} \otimes \mathbb{C}^{n})$, where  $X=(x_{i j})_{n\times n}$ is the matrix of the generators $x_{ij}$, $X_{1}=X \otimes I$  and  $X_{2}=I \otimes X$.
In terms of entries, the relations can be written as
\begin{align}\label{e:reln}
x_{i l} x_{i k}&=qx_{i k} x_{i l}, \\
x_{j k} x_{i k}&=q x_{i k} x_{j k}, \\
x_{j k} x_{i l}&=x_{i l} x_{j k}, \\
x_{j l} x_{i k}&=x_{i k} x_{j l}+(q-q^{-1}) x_{i l} x_{j k},
\end{align}
where  $i<j$  and  $k<l$.
The
algebra  $A_q(Mat_n)$ is a bialgebra under the comultiplication
 $A_q(Mat_n)\longrightarrow A_q(Mat_n)\ot A_q(Mat_n)$
 defined by
\begin{equation}
\Delta(x_{ij})=\sum_{k=1}^{n} x_{ik}\otimes x_{kj},
\end{equation}
and the counit given by $\varepsilon(x_{ij})=\delta_{ij}$. We will briefly write the coproduct
as $\Delta(X)=X\dot{\ot}X$.

\subsection{Hecke algebra $\mathcal H_m$ and the Yang-Baxter equation}
The  Hecke algebra  $\Hc_m$ is the associative algebra generated by elements
$T_1,\ldots,T_{m-1}$ subject to the relations
\begin{align}\label{e:Hecke}
&(T_i-q)(T_i+q^{-1})=0,\\
&T_{i}T_{i+1}T_{i}=T_{i+1}T_{i}T_{i+1},\\
&T_iT_j=T_jT_i \quad\text{for\ \  $|i-j|>1$}.
\end{align}

For each $1\leq i\leq m-1$, let $\si_{i}=(i, i+1)$ be the adjacent transposition in the symmetric group $\Sym_{m}$.
Choose a reduced decomposition
$\si=\si_{i_{1}} \dots \si_{i_{l}}$ of any element
$\si \in \Sym_{m}$ and set $T_{\si}=T_{i_{1}} \dots T_{i_{l}}$. It is well-known that $T_{\sigma}$ is independent of the choice of reduced expressions of $\sigma$, and
forms a linear basis of the algebra $\Hc_{m}$. 

Let  $P=\sum_{i, j=1}^{n} e_{i j} \otimes e_{j i} \in \mathrm{End}(\mathbb{C}^{n} \otimes \mathbb{C}^{n})$  be the permutation operator and  define $\check{R}=PR$. Explicitly
\begin{equation}\label{PR}
  \check{R}=q\sum_{i=1}^{n}e_{ii}\ot e_{ii}+\sum_{ i\neq j} e_{ij}\ot e_{ji}+(q-q^{-1})\sum_{ i>j} e_{ii}\ot e_{jj}.
\end{equation}
The matrix $\check{R}$ satisfies the following relations:
\ben
\bal
&\check{R}_{k}\check{R}_{k+1}\check{R}_{k}=\check{R}_{k+1}\check{R}_{k}\check{R}_{k+1},\\
&(\check{R}_{k}-q)(\check{R}_{k}+q^{-1})=0,
\eal
\een
where $\check{R}_{k}=P_{k,k+1}R_{k,k+1}$.
Then there is a representation of the Hecke algebra $\Hc_{m}$ on the
tensor product space  $(\CC^{n})^{\ot m}$ defined by
\beq
T_{k}\mapsto \check{R}_{k},\qquad k=1,\dots,m-1.
\eeq
And the RTT relation \eqref{RT} can be rewritten equivalently as
\beq
\check{R}X_1X_2=X_1X_2 \check{R}.\\
\eeq

For any integer $n\geq 0$, set
\[
[n]_q=\frac{q^n-q^{-n}}{q-q^{-1}}, \qquad  (n)_q=\frac{q^n-1}{q-1}.
\]
The quantum factorials are defined as
\[
 [n]_q!=[1]_q\cdots[n]_q, \qquad (n)_q!=(1)_q\cdots(n)_q.
\]

\subsection{Characters of $\mathcal H_m$}\label{S:HeckeChar} For generic $q$, the Hecke algebra $\mathcal{H}_m$ is semisimple and  the irreducible representations of $\mathcal{H}_m$ over $\mathbb{C}$ are parameterized by partitions of $m$. Given a partition $\lambda \vdash m$, we denote the corresponding irreducible representation by $V^{\lambda}$.

The module $V^{\lambda}$ can be realized by standard Young tableaux  of shape $\lambda\vdash m$ \cite{DJ,DJ2,IO,Mu}. A standard Young tableau  $\mathcal{T}$ of shape $\lambda$ is an assignment from $\{1, 2, \ldots, m\}$
to the boxes of the Young diagram of $\lambda$ such that the numbers are
strictly increasing across the rows and down the columns.
Let $\{v_{\mathcal T}\}$ be the normalized orthogonal Young basis of $V^{\lambda}$ indexed by standard Young tableaux of shape $\lambda$. The inner product $\langle \ , \ \rangle$ satisfies the property
that $\langle T_{\sigma}u, v\rangle=\langle u, T_{\sigma^{-1}}v\rangle$.

Let $c_k(\Tc)=j-i$ denote the content of the box $(i,j)$ of $\lambda$ occupied by $k$ in $\Tc$. Then the action of $\mathcal H_m$ is given by
\begin{align}\label{Young-orthogonal-form}
T_iv_{\Tc}=\frac{q^{d_i(\Tc)}}{[d_i(\Tc)]_q}v_{\Tc}+
\frac{\sqrt{[d_i(\Tc)+1]_q[d_i(\Tc)-1]_q}}{[d_i(\Tc)]_q}v_{s_i\Tc}
\end{align}
where $d_i(\Tc)=c_{i+1}(\Tc)-c_i(\Tc)$.

For any $u \in \mathrm{End}_{\mathbb{C}}(V^{\lambda})$, define the operator $I(u)\in \mathrm{End}_{\mathbb{C}}(V^{\lambda})$ by
\[
I(u)(v)=\sum_{\si \in \mathfrak{S}_m}T_{\si}uT_{\si^{-1}}v,
\]
where $v \in V^{\lambda}$. Then the operator $I(u)$ belongs to $\mathrm{End}_{\mathcal{H}_m}(V^{\lambda})$. It follows from the Schur lemma that
the operator $I(u)$ is a scalar multiplication of $V^{\lambda}$. Therefore for any $u \in \mathrm{End}_{\mathbb{C}}(V^{\lambda})$,
\begin{equation}\label{operater Iu}
I(u)=c_{\lambda}tr_{V^{\lambda}}(u)id_{V^{\lambda}},
\end{equation}
where  $c_{\lambda}$ is   the Schur element and given  by Steinberg's formula   \cite{Me},
 \beq
 c_{\lambda}=\prod_{\alpha\in \lambda}q^{c(\alpha)}[h(\alpha)]_q.
 \eeq
By the general properties of symmetric algebras \cite[Corollary 4.51]{Me},
the primitive idempotent $\mathcal{E}_{\Tc}^{\lambda}$  of $\mathcal{H}_m$ associated to standard Young tableau of shape $\lambda$ can be computed by
\begin{equation}\label{e:idem}
\mathcal{E}_{\Tc}^{\lambda}=\frac{1}{c_{\lambda}}\sum_{\si \in \Sym_m}\langle T_{\si^{-1}}v_{\Tc},v_{\Tc}\rangle T_{\si}.
\end{equation}
Denote by $\mathrm{SYT}(\lambda)$ the set of all standard Young tableaux of shape $\lambda$. The {\it character} of the representation $V^{\lambda}$ is defined as
\begin{equation}\label{e:char}
\chi_q^{\lambda}
=\sum_{\sigma\in \mathfrak S_m} \sum_{\Tc\in \mathrm{SYT}(\lambda)}\langle T_{\si^{-1}}v_{\Tc},v_{\Tc}\rangle  T_{\si}
=c_{\lambda}\sum_{\Tc}\mathcal{E}_{\Tc}^{\lambda}.
\end{equation}
 One can extend the character additively for any representation of $\mathcal H_m$. For irreducible characters of $\mathcal H_m$ as a trace function, see
 \cite{R}.
We will briefly write $\mathcal{E}_{\Tc}=\mathcal{E}_{\Tc}^{\lambda}$ if the shape of $\Tc$ is clear from the context.

\begin{proposition}\label{ch formula}
Let $\lambda=(\lambda_1,\dots,\lambda_l)\vdash m$ and $\Tc$ is a standard Young tableau of shape $\lambda$, then
\begin{equation}\label{irr character}
\chi_q^{\lambda}=\sum_{\si \in \mathfrak{S}_m}T_{\si}\mathcal{E}_{\Tc}^{\lambda}T_{\si^{-1}}.
\end{equation}
\end{proposition}
\begin{proof}
Let $(\rho_{\lambda},V^{\lambda})$ be the representation of $\mathcal{H}_m$ corresponding to $\lambda$.
By equation \eqref{operater Iu} and \eqref{e:char}, we have that
\begin{equation}
  \begin{aligned}
     \sum_{\si \in \mathfrak{S}_m}\rho_{\lambda}(T_{\si}\mathcal{E}_{\Tc}^{\lambda}T_{\si^{-1}})=c_{\lambda}tr_{V^{\lambda}}(\mathcal{E}_{\Tc}^{\lambda})id_{V^{\lambda}}=c_{\lambda}id_{V^{\lambda}}=c_{\lambda}\rho_{\lambda}(\sum_{\Tc}\mathcal{E}_{\Tc}^{\lambda})=\rho_{\lambda}(\chi^{\lambda}_q).
  \end{aligned}
\end{equation}
Moreover, for any other irreducible representation $(\rho_{\mu},V^{\mu})$, we have that
\begin{equation}
  \begin{aligned}
     \sum_{\si \in \mathfrak{S}_m}\rho_{\mu}(T_{\si}\mathcal{E}_{\Tc}^{\lambda}T_{\si^{-1}})=c_{\mu}tr_{V^{\mu}}(\mathcal{E}_{\Tc}^{\lambda})id_{V^{\mu}}=0,\\
     \rho_{\mu}(\chi^{\lambda}_{q})=c_{\mu}\sum_{\Tc}\rho_{\mu}(\mathcal{E}_{\Tc}^{\lambda})=0.
  \end{aligned}
\end{equation}
Hence,
\begin{equation}
\chi_q^{\lambda}=\sum_{\si \in \mathfrak{S}_m}T_{\si}\mathcal{E}_{\Tc}^{\lambda}T_{\si^{-1}}.
\end{equation}
\end{proof}

The Jucys-Murphy elements $y_1,\ldots,y_m$ of $\Hc_m$ are defined by $y_1=1$ and
\begin{align}
  y_k=1+(q-q^{-1})(T_{(1,k)}+T_{(2,k)}+\cdots+T_{(k-1,k)}), \quad k=2,\ldots,m.
\end{align}
It is well-known that these elements generate the maximal commutative subalgebra of $\Hc_m$ and
\begin{align}\label{JM_rela}
  y_k\mathcal{E}_{\Tc}^{\lambda}=\mathcal{E}_{\Tc}^{\lambda}y_k=q^{2c_k(\Tc)}\mathcal{E}_{\Tc}^{\lambda}.
\end{align}
And the primitive idempotents $\mathcal{E}_{\Tc}^{\lambda}$  can be expressed explicitly in terms of the pairwise commuting Jucys-Murphy elements of $\Hc_m$. Denote by $\Tc^-$ the standard tableau obtained from $\Tc$ by removing the box $\alpha$ occupied by $m$. The shape of $\Tc^-$ is denoted by $\mu$. We have the recurrence relation in \cite{DJ2}:
\ben
\mathcal{E}_{\Tc}^{\lambda}=\mathcal{E}_{\Tc^-}^{\mu}\frac{(y_m-q^{2a_1})\cdots (y_m-q^{2a_l})}{(q^{2c_{\alpha}}-q^{2a_1})\cdots (q^{2c_{\alpha}}-q^{2a_l})},
\een
where $a_1,\ldots,a_l$ are the contents of all addable boxes of $\mu$ except for $\alpha$, while $c_{\alpha}$ is the content of the latter.

An alternative way to express the primitive idempotents $\mathcal{E}_{\Tc}^{\lambda}$ is provided by the ${\it fusion\  procedure}$ for the Hecke algebra  $\Hc_m$ in \cite{Ch, IMO,N}.
 Take $m$ complex variables $z_1,\ldots,z_m$ and consider the $\Hc_m$-valued rational function defined by
\ben
\phi_{\Tc}(z_1,\ldots,z_m)=\prod^{\rightarrow}_{(i,j)} T_{j-i}(q^{2c_i(\Tc)}z_i,q^{2c_j(\Tc)}z_j),
\een
where the product is taken in the lexicographical order on the set of pairs $(i,j)$ with $1 \leq i <j \leq m$. And for $k=1,\ldots,m-1$,  the rational functions $T_k(x,y)$ in two variables $x,y$ be defined by
\ben
T_k(x,y)=T_k+\frac{q-q^{-1}}{x^{-1}y-1}.
\een
Then the primitive idempotent $\Ec_{\Tc}^{\lambda}$ can be obtained by the consecutive evaluations
\beq\label{fusion-procedure}
\Ec_{\Tc}^{\lambda}=\frac{1}{c_{\lambda^T}}\phi_{\Tc}(z_1,\ldots,z_m)T_0^{-1}|_{z_1=1}\cdots |_{z_m=1},
\eeq
where we let $T_0$ denote the unique longest element in $\Hc_m$ and $T_0$ satisfies the relations:
\beq\label{T0-relation}
T_0T_j=T_{m-j}T_0, \ \ 1 \leq j\leq m-1.
\eeq

For $\mu\vdash m$, let $\mathfrak S_{\mu}$ be a Young subgroup of the symmetric group $\mathfrak{S}_m$ of type $\mu$. Denote by $K$ the set of minimal representatives of cosets in $\mathfrak{S}_m/\mathfrak S_{\mu}$. Let $\mathcal{H}_{\mu}$ be the corresponding  parabolic subalgebra of $\mathcal{H}_m$, and $V$ is a representation of $\mathcal{H}_{\mu}$.
The induced representation of $V$ is defined by
$$\mathrm{Ind}_{\mathcal{H}_{\mu}}^{\mathcal{H}_m}(V)=\mathcal{H}_m\ot_{\mathcal{H}_{\mu}} V.$$
We briefly write $\mathrm{Ind}(V)=\mathrm{Ind}_{\mathcal{H}_{\mu}}^{\mathcal{H}_m}(V)$.
 \begin{proposition}\label{induced-pro}
  Let $(\rho,V)$ be an irreducible representation of $\mathcal{H}_{\mu}$ with character  $\chi^{V}$, then the character $\chi^{\mathrm{Ind}(V)}$ of the induced representation $\mathrm{Ind}(V)$ is given by
  \begin{equation}\label{induced-character}
    \chi^{\mathrm{Ind}(V)}=\frac{1}{c_V}\sum_{\si \in\mathfrak{S}_m}T_{\si}\chi^{V}T_{\si^{-1}},
  \end{equation}
  where ${c_V}$ is the Schur element of $V$.
\end{proposition}

\begin{proof} Using the property \eqref{operater Iu} in $\mathrm{End}_{\mathcal H_{\mu}}(V)$, we see that
\begin{equation}
  \begin{aligned}
     \sum_{\tau \in \mathfrak S_{\mu}}\rho(T_{\tau}\chi^{V}T_{\tau^{-1}})=c_{V}tr_{V}(\chi^{V})id_{V}=c_V\rho(\chi^{V}).
  \end{aligned}
\end{equation}
Moreover, for any other irreducible representation $(\eta,V')$, we have that
\begin{equation}
  \begin{aligned}
     \sum_{\tau \in \mathfrak S_{\mu}}\eta(T_{\tau}\chi^{V}T_{\tau^{-1}})&=c_{V'}tr_{V'}(\chi^{V})id_{V'}\\
     &=c_{V'}\sum_{\si\in \mathfrak{S}_{\mu}}\sum_{i,j}\rho_{ii}(T_{\si})\eta_{jj}(T_{\si^{-1}})\\
     &=\delta_{\rho,\eta}c_{V}\dim(V)\\&=0.
  \end{aligned}
\end{equation}
By the property of semisimplicity and  the regular representation of $\mathcal{H}_n$, we have that
\begin{equation}
\sum_{\tau \in \mathfrak S_{\mu}}T_{\tau}\chi^{V}T_{\tau^{-1}}=c_V \chi^{V}.
\end{equation}
Thus
\begin{equation}
  \begin{aligned}
     \sum_{\si \in\mathfrak{S}_m}T_{\si}\chi^{V}T_{\si^{-1}}&=\sum_{\si \in K}\sum_{\tau \in \mathfrak S_{\mu}}T_{\si}T_{\tau}\chi^{V}T_{\tau^{-1}}T_{\si^{-1}}\\
     &=\sum_{\si \in K}T_{\si} (\sum_{\tau \in \mathfrak S_{\mu}}T_{\tau}\chi^{V}T_{\tau^{-1}}) T_{\si^{-1}}\\
     &=c_{V}\sum_{\si \in K}T_{\si} \chi^{V} T_{\si^{-1}}.
  \end{aligned}
\end{equation}

Let $\chi^{V}=c_{V}\sum_{i=1}^{d_{V}}\mathcal{E}_{\Tc_i}$. Then $\mathrm{Ind}(V)=span_{\mathbb{C}[q,q^{-1}]}\{T_{\si}\mathcal{E}_{\Tc_i}\mid \si \in K, 1 \leq i \leq d_{V}\}$, so it is sufficient to prove that for any $\si,\tau\in K, 1 \leq i \leq d_{V} $ and $T_{\eta}\in \mathcal{H}_m$, the coefficients of $T_{\si}\mathcal{E}_{\Tc_i}$ in $T_{\eta}T_{\tau}\mathcal{E}_{\Tc_i}$ are equal to the coefficients of $T_{\eta}$ in $c_VT_{\si}\mathcal{E}_{\Tc_i} T_{\tau^{-1}}$.

Indeed, the coefficients of $T_{\si}\mathcal{E}_{\Tc_i}$ in $T_{\eta}T_{\tau}\mathcal{E}_{\Tc_i}$ are equal to the coefficients of $T_{\si}$ in $c_VT_{\eta}T_{\tau}\mathcal{E}_{\Tc_i}$. And
\begin{equation}
  c_VT_{\eta}T_{\tau}\mathcal{E}_{\Tc_i}=\sum_{\si\in \mathfrak{S}_m}\sum_{\upsilon \in \mathfrak S_{\mu}}c_V\mathcal{E}_{\Tc_i}(T_{\upsilon})c_{\eta,\tau\upsilon}^{\si}T_{\si},
\end{equation}
where there exists elements $c_{\eta,\tau\upsilon}^{\si} \in \mathbb{C}[q,q^{-1}]$ occurring as coefficients in the decomposition.

On the other hand,
\begin{equation}
\begin{aligned}
 c_VT_{\si}\mathcal{E}_{\Tc_i} T_{\tau^{-1}}&=\sum_{\upsilon\in \mathfrak S_{\mu}}\mathcal{E}_{\Tc_i}(T_{\upsilon^{-1}})T_{\si}T_{\upsilon^{-1}}T_{\tau^{-1}} \\
  &=\sum_{\eta\in \mathfrak{S}_m}\sum_{\upsilon\in \mathfrak S_{\mu}}c_V\mathcal{E}_{\Tc_i}(T_{\upsilon})c_{\si,\upsilon^{-1}\tau^{-1}}^{\eta}T_{\eta}.
\end{aligned}
\end{equation}
Since $c_{\si,\upsilon^{-1}\tau^{-1}}^{\eta}=c_{\upsilon^{-1}\tau^{-1},\eta^{-1}}^{\si^{-1}}=c_{\eta,\tau\upsilon}^{\si}$ \cite[Lemma 4.1]{KS}, so the coefficients of $T_{\si}\mathcal{E}_{\Tc_i}$ in $T_{\eta}T_{\tau}\mathcal{E}_{\Tc_i}$ are equal to the coefficients of $T_{\eta}$ in $c_VT_{\si}\mathcal{E}_{\Tc_i} T_{\tau^{-1}}$.
Hence $$\chi^{\mathrm{Ind}(V)}=\sum_{\si \in K}T_{\si} \chi^{V} T_{\si^{-1}}$$ and $$\chi^{\mathrm{Ind}(V)}=\frac{1}{c_V}\sum_{\si \in\mathfrak{S}_m}T_{\si}\chi^{V}T_{\si^{-1}}.$$
\end{proof}

\section{Quantum immanants and Schur-Weyl-Jimbo duality}
In this section, we will introduce the quantum immanants of the quantum coordinate algebra $A_q(\Mat_n)$ for any  representation  of the
Hecke algebra $\Hc_m$ and give a representation-theoretic interpretation for quantum immanants.
\subsection{Quantum immanants}
Let  $I=(i_1,\ldots, i_m)$ and $J=(j_1, \ldots ,j_m)$ be two multisets ($m$-tuples) of $[n]=\{1,2,\dots,n\}$. Let $X=(x_{ij})_{n\times n}$ be the generator matrix of the quantum coordinate algebra $A_q(\Mat_n)$. We denote by $X^I_J$ the generalized submatrix of $X$ whose row (resp. column) indices belong to  $I$ (resp. $J$). If $I=J$, we briefly write $X_I$.

It will be convenient to use a standard notation for the elements  $A_{j_1,\dots,j_m}^{i_1,\dots,i_m}\in A_{q}(\Mat_{n})$ of an operator $A$ in $A_{q}(\Mat_{n}) \otimes \mathrm{End}((\mathbb{C}^{n})^{\ot m} )$.
We denote

\begin{equation}
  A=\sum_{I,J}A_{j_1,\dots,j_m}^{i_1,\dots,i_m}\ot e_{i_1j_1}\ot \cdots \ot e_{i_mj_m}.
\end{equation}
Let $\{e_{i}^*|1\leq i\leq n\}$ be the basis of $\mathbb (\mathbb C^n)^*$ dual to the basis $\{e_{i}|1\leq i\leq n\}$ of $\mathbb C^n$, then we can write the dual bases of $((\mathbb C^{n})^{\ot m})^*$ and $(\mathbb C^{n})^{\ot m}$ respectively:
\beq
\langle i_1,\dots ,i_m\mid=e_{i_1}^*\otimes\cdots \otimes  e_{i_m}^*, \qquad \mid i_1,\dots,i_m\rangle =e_{i_1}\otimes\cdots \otimes e_{i_m}.
\eeq
Then the matrix coefficients of the operator $A$ are given by
$$A_{j_1,\dots,j_m}^{i_1,\dots,i_m}=\langle i_1,\dots ,i_m\mid A \mid j_1,\dots,j_m\rangle \in A_q(\Mat_n)$$

Let $V$ be any representation of $\Hc_m$ with the character $\chi ^{V}$.
The {\it quantum immanant} of $X^I_J$ associated to the representation $V$ is defined by
\begin{equation}\label{imm-def}
  \Imm_{\chi ^{V}}(X^I_J)=\langle i_{1},\ldots ,i_{m}\mid \chi^{V}X_1\cdots X_{m} \mid j_{1},\ldots,j_{m}\rangle.
\end{equation}
In particular, $m=n$, the quantum immanant becomes the quantum determinant $\det_q(X)$ when $\lambda=(1^n)$ and the quantum permanent $\mathrm{per}_q(X)$ when $\lambda=(n)$.
It degenerates to Littlewood's immanant \cite{Li,LR} when $q\mapsto 1$:
\begin{equation}
\lim_{q\to 1} \Imm_{\chi_{q}^{\lambda}}(X_I)=\sum_{\sigma\in \mathfrak S_m} \chi^{\lambda}(\sigma)x_{i_{\sigma(1)}i_1}x_{i_{\sigma(2)}i_2}\cdots x_{i_{\sigma(m)}i_m}.
\end{equation}

With Proposition \ref{ch formula}, the quantum immanants corresponding to representation $V^{\lambda}$ can be written as
\begin{equation}
  \Imm_{\chi_{q}^{\lambda}}(X^I_J)=\sum_{\si\in \Sym_m}\langle i_{1},\ldots ,i_{m} \mid \check{R}_{\si} \mathcal{E}_{\Tc}^{\lambda}\check{R}_{\si^{-1}} X_1\cdots X_{m} \mid j_{1},\ldots,j_{m}\rangle.
\end{equation}

It follows from Equation \eqref{PR}  that $\check{R}(v_i\ot v_j)=v_j\ot v_i$ for $i<j$. Therefore, for any $\si \in \mathfrak{S}_m$ and subset $I=(1\leq i_1 < \ldots < i_m\leq n)$, we have
\begin{equation}
\begin{aligned}
  \check{R}_{\si^{-1}} \mid i_1,\dots,i_m\rangle=\mid i_{\si_1},\dots i_{\si_m}\rangle, \\
\langle i_1,\dots,i_m \mid \check{R}_{\si}=\langle i_{\si_1},\dots i_{\si_m}\mid.
\end{aligned}
\end{equation}

It is convenient to write $I=(1^{m_1},\ldots,n^{m_n})$ to specify the multiplicity $m_i$ of $i$ in the non-decreasing multiset $I=(i_1\leq \ldots \leq i_m)$. Here $m_i=m_i(I)= \mathrm{Card}\{j\in I |j=i\}$.

Let  $$m({I})=m_1!m_2 ! \cdots m_n !$$ and $$m_q({I})=(m_1)_q!(m_2)_q ! \cdots (m_n)_q !$$ Denote by $$\mathfrak S_{I}=\mathfrak S_{m_1}\times\mathfrak S_{m_2}\times \cdots \times \mathfrak S_{m_n}$$  the Young subgroup  of $\mathfrak{S}_m$, and $\mathfrak S_{m_j}=1$ if $m_j=0$.
\begin{proposition}\label{imm-char}
\label{X_I}
  Let $I=(i_1\leq \ldots \leq i_m)$ be an ordered  multiset of $[n]$. For any $\lambda\vdash m$,
let  $\mathcal{E}_{\Tc}^{\lambda}$ be the primitive idempotent   of $\mathcal{H}_m$ associated to standard Young tableau ${\Tc}$ of shape $\lambda$.
  Then
  \begin{equation}\label{imm-XI}
    \frac{\Imm_{\chi_q^{\lambda}}(X_I)}{m_{q^2}(I)}=\frac{1}{m(I)}\sum_{\si \in \mathfrak{S}_m}\langle i_{\si_1},\ldots ,i_{\si_m}\mid \mathcal{E}_{\Tc}^{\lambda} X_1\cdots X_m \mid i_{\si_1},\ldots ,i_{\si_m}\rangle.
  \end{equation}
\end{proposition}
\begin{proof} We denote $\mathcal{E}_{\Tc}^{\lambda}$ by $\mathcal{E}_{\Tc}$.
By definition,
\begin{equation}
\bal
  \Imm_{\chi_q^{\lambda}}(X_I)&=\sum_{\sigma\in \mathfrak{S}_m}\langle i_{1},\ldots ,i_{m}\mid \check{R}_{\si} \mathcal{E}_{\Tc}\check{R}_{\si^{-1}}X_1\cdots X_{m} \mid i_{1},\ldots,i_{m}\rangle\\
  &=\sum_{\sigma\in \mathfrak{S}_m}\langle i_{1},\ldots ,i_{m}\mid \check{R}_{\si} \mathcal{E}_{\Tc}X_1\cdots X_{m} \check{R}_{\si^{-1}}\mid i_{1},\ldots,i_{m}\rangle.
\eal
\end{equation}

Let $\mathcal M(\Sym_m/\Sym_I)$ be the minimal length coset representative of $\Sym_m/\Sym_I$.
For any  $\tau \in\Sym_{I}$,
\begin{align*}
\check{R}_{\tau}\mid i_1,\dots ,i_m\rangle=q^{l(\tau)}\mid i_1,\dots ,i_m\rangle,\\
 \langle i_1,\dots ,i_m \mid \check{R}_{\tau^{-1}}= q^{l(\tau)} \langle i_1,\dots ,i_m \mid.
\end{align*}
For any  $\mu \in \mathcal M(\Sym_m/\Sym_I)$,
\begin{align*}
  &\check{R}_{\mu}  \mid i_1,\dots ,i_m\rangle= \mid i_{\mu_1},\dots ,i_{\mu_m}\rangle,\\
  & \langle i_1,\dots ,i_m\mid \check{R}_{\mu^{-1}}= \mid i_{\mu_1},\dots ,i_{\mu_m}\rangle.
\end{align*}
Therefore,
\begin{equation*}
\bal
  \Imm_{\chi_q^{\lambda}}(X_I)&=\sum_{ \tau \in \Sym_{I}, \atop \mu \in \mathcal M(\Sym_m/\Sym_I) } \langle i_1,\dots ,i_m\mid  \check{R}_{\tau^{-1}} \check{R}_{\mu^{-1}}\mathcal{E}_{\Tc}X_1\cdots X_{m}\check{R}_{\mu} \check{R}_{\tau} \mid i_1,\dots ,i_m\rangle\\
  &=\frac{m_{q^2}(I) }{m( {I})}\sum_{\si\in \Sym_m}\langle i_{\sigma_1},\dots ,i_{\sigma_m}\mid  \mathcal{E}_{\Tc}X_1\cdots X_{m} \mid i_{\sigma_1},\dots ,i_{\sigma_m}\rangle.
\eal
\end{equation*}
This completes the proof.
\end{proof}

\subsection{Gelfand-Tsetlin basis of  $U_q(\mathfrak{gl}_n)$ }

We review that the quantum universal enveloping algebra $U_q(\mathfrak{gl}_n)$ is the unital associative complex algebra with generators $E_i,F_i (1\leq i\leq n-1)$ and $q^{\pm\frac{\varepsilon_i}{2}} (1\leq i\leq n)$ with the following relations:
\beq
\bal
&q^{\frac{\varepsilon_i}{2}}q^{-\frac{\varepsilon_i}{2}}
=q^{-\frac{\varepsilon_i}{2}}q^{\frac{\varepsilon_i}{2}}=1, \qquad q^{\frac{\varepsilon_i}{2}}q^{\frac{\varepsilon_j}{2}}
=q^{\frac{\varepsilon_j}{2}}q^{\frac{\varepsilon_i}{2}},\\
&q^{\frac{\varepsilon_i}{2}}E_jq^{-\frac{\varepsilon_i}{2}}=q^{\frac{\delta_{ij}-\delta_{i-1,j}}{2}}E_j,
\qquad
q^{\frac{\varepsilon_i}{2}}F_{j}q^{-\frac{\varepsilon_i}{2}}=q^{\frac{-\delta_{ij}+\delta_{i-1,j}}{2}}F_{j},\\
&E_i^2E_{j}-(q+q^{-1})E_i E_{j} E_i+E_{j}E_i^2=0,\ (|i-j|=1)\\
&F_i^2F_{j}-(q+q^{-1})F_i F_{j} F_i+F_{j}F_i^2=0,\ (|i-j|=1)\\
&E_i E_j= E_j E_i, \qquad F_i F_j= F_j F_i, \ (|i-j|>1),\\
&[E_i,F_j]=\delta_{ij}\frac{q^{\varepsilon_i-\varepsilon_{i+1}}-q^{-\varepsilon_i+\varepsilon_{i+1}}}{q-q^{-1}}.
\eal
\eeq
The algebra $U_q(\mathfrak{gl}_n)$ has the following alternative presentation by $L$-operators. It is generated by the entries of two triangular matrices $L^+=(l_{ij}^+)$ and $L^-=(l_{ij}^-)$ subject to the defining relations
\beq
\bal
l_{ij}^+&=l_{ji}^-=0, \ \ \ 1 \leq j < i \leq n,\\
l_{ii}^+ l_{ii}^-&=l_{ii}^- l_{ii}^+=1, \ \ \ 1 \leq i \leq n,\\
R^+L_1^{\pm}L_2^{\pm}&=L_2^{\pm}L_1^{\pm}R^+,\\
R^+L_1^{+}L_2^{-}&=L_2^{-}L_1^{+}R^+.
\eal
\eeq
The isomorphism between these two presentations is given by
\beq
\bal
l_{ij}^+=\left\{\begin{array}{ccc}
q^{\varepsilon_{i}}, \ &i=j\\
(q-q^{-1})q^{-\frac{1}{2}+\frac{\varepsilon_{i}+\varepsilon_{j}}{2}}\hat{E}_{ji}, \ &i>j\\
0, \ &i<j\\
\end{array}\right.
\eal
\eeq
and
\beq
\bal
l_{ij}^-=\left\{\begin{array}{ccc}
q^{-\varepsilon_{i}}, \ &i=j\\
-(q-q^{-1})q^{\frac{1}{2}-\frac{\varepsilon_{i}+\varepsilon_{j}}{2}}\hat{E}_{ji}, \ &i<j\\
0, \ &i>j.\\
\end{array}\right.
\eal
\eeq
where   $\hat{E}_{ij} (i \neq j)$ are determined by the following equations \cite{Ji}
\beq
\bal
&\hat{E}_{i,i+1}=E_i,\ \ \ \hat{E}_{i+1,i}=F_i,\\
&\hat{E}_{ij}=\hat{E}_{ik}\hat{E}_{kj}-q^{\pm 1}\hat{E}_{kj}\hat{E}_{ik}, \ \ (i \gtrless k \gtrless j).
\eal
\eeq
Define the matrix with spectral parameter $u$
\beq
L(u)=u L^{+}-u^{-1}L^{-}.
\eeq

Let $\Lambda=(\lambda_{ij})_{1 \leq i \leq n\atop 1\leq j \leq i}$  be  Gelfand-Tsetlin patterns attached to simple finite dimensional
$U_q(\mathfrak{gl}_n)$-module $U^{\lambda}$
by the Gelfand-Tsetlin formulae \cite{GT,UTS1,UTS2}, the Gelfand-Tsetlin patterns $\Lambda$ with respect to the basis vectors of $U^{\lambda}$ satisfy the relations:
\beq
\lambda_{ij}\geq \lambda_{i-1,j}\geq \lambda_{i,j+1}, \ \ 1\leq j \leq i \leq n.
\eeq
Denote $\{\xi_{\Lambda}\}$  are the Gelfand-Tsetlin  basis vectors in $U^{\lambda}$.

Recall the central element $z_n(u)$ of $U_q(\mathfrak{gl}_n)$ introduced by Jimbo \cite{Ji}:
\beq
\bal
z_n(u)=(q-q^{-1})^{-n}\sum_{\si \in \Sym_n}(-q)^{-l(\si)}L(u q^{-n+1})_{\si_n,n}\cdots L(u)_{\si_1,1}.
\eal
\eeq
Let $Z_k$ be the center of $U_q(\mathfrak{gl}_k)$. The subalgebra of $U_q(\mathfrak{gl}_n)$ generated by $\{Z_k \mid k=1 ,\ldots,n\}$ is called the Gelfand–Tsetlin subalgebra of $U_q(\mathfrak{gl}_n)$ and is denoted by $\Gamma_q$.
\begin{lemma}\label{central-char}
The  element $z_k(u)$ of  $\Gamma_q$ acts on the Gelfand-Tsetlin basis vector $\xi_{\Lambda}$ in simple $U_q(\mathfrak{gl}_n)$-module $U^{\lambda}$ as a scalar multiplication by
\beq
 \prod_{l=1}^{k}\frac{uq^{\lambda_{kl}-l+1}-u^{-1}q^{-\lambda_{kl}+l-1}}{q-q^{-1}}.
\eeq
\end{lemma}
\begin{proof}
Let $W$ be the subspace of $U^{\lambda}$ spanned by all  $\xi_{\Lambda'}$ such that $\lambda'_{lj}=\lambda_{lj}$ for $k \leq l \leq n, \ 1 \leq j \leq l$. Then $W$ is a finite dimensional simple module of $U_q(\mathfrak{gl}_k)$ with highest weight $\mu$ where $\mu_i=\lambda_{ki}$.

Let $\Lambda_0$ be the Gelfand-Tsetlin pattern with $\lambda_{ij}=\lambda_{kj}$ for any $1 \leq j \leq i \leq k$. Then  $\Lambda_0$ be the highest weight vector in $W$.  Note that $l_{ij}^{+}\cdot \xi_{\Lambda_0}=0, 1 \leq j\neq i\leq k$. Then the element $L(\lambda q^{-k+1})_{\si_k,k}\cdots L(\lambda)_{\si_1,1}$ kills $v$ unless $\si_1=1, \ldots, \si_k=k$. We thus have
\beq
\bal
  z_k(u)\cdot \xi_{\Lambda_0} &=(q-q^{-1})^{-k}L(u q^{-k+1})_{k,k}\cdots L(u)_{1,1}\\
  &=\prod_{l=1}^{k}\frac{uq^{\lambda_{kl}-l+1}-u^{-1}q^{-\lambda_{kl}+l-1}}{q-q^{-1}}.
\eal
\eeq
Since $W$ is a finite-dimensional simple module of $U_q(\mathfrak{gl}_k)$, the action of $z_k(u)$ on any vector in $W$ is given as above.
\end{proof}
\subsection{Schur-Weyl-Jimbo duality}

We will give a weight space representation for quantum immanants of quantum coordinate algebra $A_q(\Mat_n)$ using the Schur-Weyl-Jimbo duality.
We can define the action of the quantum group $U_q(\mathfrak{gl}_n)$ on tensors in $\mathbb{C}^{n}$ by the Hopf algebra structure of $U_q(\mathfrak{gl}_n)$:
\beq
\bal
&\Delta(q^{\pm \frac{\varepsilon_i}{2}})=q^{\pm \frac{\varepsilon_i}{2}} \ot q^{\pm \frac{\varepsilon_i}{2}},\\
&\Delta(E_i)=E_i\ot q^{ \frac{(\varepsilon_i-\varepsilon_{i+1})}{2}}
+q^{ \frac{-(\varepsilon_i-\varepsilon_{i+1})}{2}} \ot E_i,\\
&\Delta(F_i)=F_i\ot q^{ \frac{(\varepsilon_i-\varepsilon_{i+1})}{2}}
+q^{ \frac{-(\varepsilon_i-\varepsilon_{i+1})}{2}} \ot F_i,\\
&\Delta(L^{\pm})=L^{\pm} \ot L^{\pm}.
\eal
\eeq

By Schur-Weyl-Jimbo duality \cite{Ji}, $(\mathbb{C}^{n})^{\ot m}$ is a multiplicity free $U_q(\mathfrak{gl}_n)\times  \Hc_m $-module and there exists a decomposition of the space of tensors
\beq\label{Schur-Weyl-decom}
(\mathbb{C}^{n})^{\ot m}\cong\bigoplus_{\lambda \vdash m, \atop \text{depth}(\lambda) \leq n} U^{\lambda}\ot V^{\lambda},
\eeq
where let $(\pi_{\lambda},U^{\lambda})$ are the simple left module of highest weight $\lambda$ for $U_q(\mathfrak{gl}_n)$ and $(\rho_{\lambda},V^{\lambda})$ are the simple left module of $\Hc_m$.

We introduce $*$-operations for $\Hc_m $  and $U_q(\mathfrak{gl}_n)$ respectively. Let  $T_{\si}^{*}=T_{\si^{-1}} \in \Hc_m$. Set $(q^{\pm \varepsilon _i})^*=q^{\pm \varepsilon _i}$, $(E_i)^*=F_i$ and $(F_j)^*=E_j$ in $U_q(\mathfrak{gl}_n)$.  They are involutive anti-automorphisms.

Define an inner product $\langle \cdot | \cdot \rangle$ on  $(\mathbb{C}^{n})^{\ot m}$ by
\beq\label{inner-product}
\bal
\langle e_{i_1}\ot \cdots \ot e_{i_m}, e_{j_1}\ot \cdots \ot e_{j_m}  \rangle&=\prod_{1\leq k \leq m}\delta_{i_k,j_k}.
\eal
\eeq
\begin{lemma}\label{inva lemma}
  The $*$-operations for $\Hc_m $  and $U_q(\mathfrak{gl}_n)$ are Hermitian adjoints on the inner product space $(\mathbb{C}^{n})^{\ot m}$.
\end{lemma}

\begin{proof}
First, to show that the  $*$-operation for $\Hc_m$ is adjoint it is sufficient to prove this for $T_k$, $1 \leq k \leq m-1$. Indeed,
\ben
\bal
&\langle \check{R}_{(k,k+1)}\cdot e_{i_1}\ot \cdots \ot e_{i_m}, e_{j_1}\ot \cdots \ot e_{j_m}  \rangle\\
&=\delta_{i_{k}\geq i_{k+1}}q^{\delta_{i_{k}= i_{k+1}}}(q-q^{-1})^{\delta_{i_{k}> i_{k+1}}}\langle e_{i_1}\ot \cdots \ot e_{i_m}, e_{j_1}\ot \cdots \ot e_{j_m}  \rangle\\
&\quad +\delta_{i_{k}\neq i_{k+1}}\langle e_{i_1}\ot \cdots \ot e_{i_{k+1}}\ot e_{i_k}\ot \cdots \ot e_{i_m}, e_{j_1}\ot \cdots \ot e_{j_m}  \rangle\\
&=\delta_{i_{k}\geq i_{k+1}}q^{\delta_{i_{k}= i_{k+1}}}(q-q^{-1})^{\delta_{i_{k}> i_{k+1}}}\prod_{1\leq t \leq m}\delta_{i_t,j_t}\\
&\quad +\delta_{i_{k}\neq i_{k+1}}(\prod_{t\neq k,k+1}\delta_{i_t,j_t})\delta_{i_{k},j_{k+1}}\delta_{i_{k+1},j_{k}}\\
&=\delta_{j_{k}\geq j_{k+1}}q^{\delta_{j_{k}= j_{k+1}}}(q-q^{-1})^{\delta_{j_{k}> j_{k+1}}}\prod_{1\leq t \leq m}\delta_{i_t,j_t}\\
&\quad +\delta_{j_{k}\neq j_{k+1}}(\prod_{t\neq k,k+1}\delta_{i_t,j_t})\delta_{i_{k},j_{k+1}}\delta_{i_{k+1},j_{k}}\\
&=\langle e_{i_1}\ot \cdots\ot e_{i_m},\check{R}_{(k,k+1)}\cdot e_{j_1}\ot \cdots \ot e_{j_m}  \rangle,
\eal
\een
Hence, for any element $\check{R}_{\si} \in \Hc_m$,
\beq
\bal
\langle \check{R}_{\si}&\cdot e_{i_1}\ot \cdots \ot e_{i_m}, e_{j_1}\ot \cdots \ot e_{j_m}  \rangle\\
&=\langle e_{i_1}\ot \cdots \ot e_{i_m}, (\check{R}_{\si})^*  \cdot e_{j_1}\ot \cdots \ot e_{j_m}  \rangle.
\eal
\eeq

On the other hand, for generators of $U_q(\mathfrak{gl}_n)$ we have that
\beq
\bal
&\langle q^{\frac{\pm \varepsilon_k}{2}}\cdot  e_{i_1}\ot \cdots \ot e_{i_m}, e_{j_1}\ot \cdots \ot e_{j_m}  \rangle\\
=&\prod_{1\leq t \leq m}q^{\frac{\pm \delta_{i_t,k}}{2}}\delta_{i_t,j_t}\\
=&\langle e_{i_1}\ot \cdots \ot e_{i_m},q^{\frac{\pm \varepsilon_k}{2}}\cdot  e_{j_1}\ot \cdots \ot e_{j_m}    \rangle.
\eal
\eeq
And
\beq
\bal
&\langle E_k\cdot e_{i_1}\ot \cdots \ot e_{i_m}, e_{j_1}\ot \cdots \ot e_{j_m}   \rangle\\
&=\sum_{l=1}^m\delta_{i_l,k+1}q^{\frac{\sum_{t \neq l}(-1)^{\delta_{t<l}}(\delta_{i_t,k}-\delta_{i_t,k+1})}{2}}\\
&\quad\langle e_{i_1}\ot \cdots \ot e_{i_{l-1}} \ot k \ot e_{i_{l+1}} \ot \cdots \ot e_{i_m}, e_{j_1}\ot \cdots \ot e_{j_m}  \rangle\\
&=\sum_{l=1}^m\delta_{i_l,k+1}q^{\frac{\sum_{t \neq l}(-1)^{\delta_{t<l}}(\delta_{i_t,k}-\delta_{i_t,k+1})}{2}}
(\prod_{s\neq l}\delta_{i_s,j_s})\delta_{j_{l},k}\\
&=\sum_{l=1}^m \delta_{j_{l},k} q^{\frac{\sum_{t \neq l} (-1)^{\delta_{t<l}} (\delta_{j_t,k}-\delta_{j_t,k+1})}{2}}
(\prod_{s\neq l}\delta_{i_s,j_s})\delta_{i_l,k+1}\\
&=\langle  e_{i_1}\ot \cdots \ot e_{i_m}, F_k\cdot e_{j_1}\ot \cdots \ot e_{j_m}   \rangle.
\eal
\eeq
\end{proof}

By Lemma \ref{inva lemma}, the basis of $U^{\lambda}\otimes V^{\lambda}$ can be expressed as the tensor product of the Gelfand-Tsetlin basis and the Young basis, both of which are orthonormal under the inner product.

Let $\{v_{\Tc} | sh(\Tc)=\lambda\}$ be
the  Young basis of $V^{\lambda}$. We can decompose \eqref{Schur-Weyl-decom} into basis vectors:
\beq\label{Schur-Weyl-decom2}
(\mathbb{C}^{n})^{\ot m}\cong\sum_{\lambda \vdash k,\atop \text{depth}(\lambda)\leq n}\sum_{(\Lambda,\Tc)}  \xi_{\Lambda}\ot v_{\Tc}.
\eeq
These basis vectors in decomposition \eqref{Schur-Weyl-decom2} are normalized orthogonal, i.e.
$$\langle \xi_{\Lambda}\ot v_{\Tc}, \xi_{\Lambda'}\ot v_{\Tc'}\rangle =\delta_{\Lambda\Lambda'}\delta_{\Tc\Tc'}$$
for any Gelfand-Tsetlin patterns $\Lambda,\Lambda'$ and standard  Young tableaux $\Tc,\Tc'$.


There exists a bialgebra pairing
$$(\ ,\ ): U_q(\mathfrak{gl}_n) \times A_q(\Mat_n)\rightarrow \mathbb{C}$$
between the two bialgebras $U_q(\mathfrak{gl}_n)$ and $A_q(\Mat_n)$.
This duality equips  $(\mathbb{C}^n)^{\ot m}$ has a $A_q(\Mat_n)$-comodule structure given by
\beq
\bal
  \Delta: (\mathbb{C}^n)^{\ot m} &\rightarrow  (\mathbb{C}^n)^{\ot m}\ot A_q(\Mat_n)\\
   e_{j_1}\ot \cdots \ot e_{j_m}&\mapsto \sum_{(i_1,\ldots,i_m)}e_{i_1}\ot \cdots \ot e_{i_k}\ot x_{i_1,j_1}\cdots x_{i_m,j_m},\\
\eal
\eeq
where the sum is over all sequence $(i_1\ldots,i_m)$ of $[n]$.

Let $\mu=(\mu_1,\ldots,\mu_m)$ be an arbitrary composition of $m$, which can also be regarded as a weight. For any representation  $W$ of  $U_q(\mathfrak{gl}_n)$, we denote by
 $W_{\mu}$ the weight space of $W$   corresponding to the weight $\mu$.
Denote by $\mathcal P_{\mu}$ the projection operator $W\rightarrow W_{\mu}$.
Now we give a weight space representation for quantum immanants of the quantum coordinate algebra $A_q(\Mat_n)$.

\begin{theorem}\label{q-Kostant-thm}
  Let $\lambda \vdash m$,  $U^{\lambda}$  be the simple finite dimensional
$U_q(\mathfrak{gl}_n)$-module with highest weight $\lambda$.  $I=(1 \leq i_1\leq \ldots \leq i_m \leq n)$ be an ordered multiset of $[n]$, $\mu$ is a composition of $m$ with $\mu_i=m_i(I)$. Then
\beq\label{general-Kostant-iden}
\frac{\Imm_{\chi_q^{\lambda}}(X_I)}{m_{q^2}(I)}=tr(\mathcal P_{\mu}\ot 1)\circ\Delta\circ \mathcal P_{\mu}|_{U^{\lambda}}.
\eeq
\end{theorem}
\begin{remark}
  For $\mu=(1^{n})$, the theorem establishes a quantum analog of Kostant's theorem in \cite{Ko}, which provides a ($\mathfrak{sl}_n$) 0-weight interpretation of the immanant.
\end{remark}

\subsection{The proof of Theorem \ref{q-Kostant-thm}}
Let $\lambda$ be a partition of $n$ and $\mu$ be a composition of $n$.
It is well-known that the semi-standard Young tableau $\Tc^{\lambda}_{\mu}$ of shape $\lambda$ and weight $\mu$, where  $\mu_i$ equals the number of $i$'s in $\Tc^{\lambda}_{\mu}$, is one-to-one corresponding  to a Gelfand-Tsetlin pattern $\Lambda=(\lambda_{ij}),1\leq j\leq i\leq n$ such that
\beq\label{pattern-tableau-corr}
\bal
\lambda_{ij}& = \text{the number of entries }\le i\text{ in the }j\text{th row of }\Tc^{\lambda}_{\mu},\\
\mu_i&=\sum_{k=1}^{i}\lambda_{ik}-\sum_{s=1}^{i-1}\lambda_{i-1,s}.
\eal
\eeq
We will briefly write $\Tc_{\mu}=\Tc^{\lambda}_{\mu}$ if the shape of $\Tc$ is clear from the context.
By corresponding \eqref{pattern-tableau-corr}, the dimensition of $(U^{\lambda})_{\mu}$  is the number of the semi-standard   tableaux $\Tc^{\lambda}_{\mu}$. Moreover, we can denote by $\{\xi_{\Tc_{\mu}}\}$ the Gelfand-Tsetlin  basis vectors in $(U^{\lambda})_{\mu}$.

There are some useful lemmas to prove theorem \ref{q-Kostant-thm} in the following.
\begin{lemma}\label{lemma1}
Let $\lambda \vdash m$ and assume $m \leq n$, then in $U_q(\mathfrak{gl}_n) \times \Hc_m $-modules,
\beq
\sqrt{c_{\lambda}} \mathcal{E}_{\Tc}^{\lambda}\cdot e_1\ot \cdots \ot e_m=\xi_{\Tc}\ot v_{\Tc} .
\eeq
\end{lemma}
\begin{proof}
By the Schur-Weyl-Jimbo duality,  $\mathcal{E}_{\Tc}^{\lambda}$ is the projection operator from $(\mathbb{C}^{n})^{\ot m}$ to $U^{\lambda}$.
\beq
\bal
 &\langle  \mathcal{E}_{\Tc}^{\lambda}\cdot e_1\ot \cdots \ot e_m, \mathcal{E}_{\Tc'}^{\lambda}\cdot e_1\ot \cdots \ot e_m\rangle\\
 &=\langle  \mathcal{E}_{\Tc'}^{\lambda}\mathcal{E}_{\Tc}^{\lambda}\cdot e_1\ot \cdots \ot e_m, e_1\ot \cdots \ot e_m\rangle\\
 &=\delta_{\Tc\Tc'}\langle  \mathcal{E}_{\Tc}^{\lambda}\cdot e_1\ot \cdots \ot e_m, e_1\ot \cdots \ot e_m\rangle\\
 &=\frac{1}{c_{\lambda}}\delta_{\Tc\Tc'}.
\eal
\eeq
So by relation \eqref{JM_rela}, we have that
\beq
\sqrt{c_{\lambda}} \mathcal{E}_{\Tc}^{\lambda}\cdot e_1\ot \cdots \ot e_m\subset U^{\lambda}\ot v_{\Tc} ,
\eeq

Moreover by Lemma \eqref{central-char}, for $m\leq  k \leq n$
\beq
\bal
&z_{k}(u)\cdot \mathcal{E}_{\Tc}^{\lambda} e_1\ot \cdots \ot e_m\\
&=\prod_{l=1}^{m}\frac{uq^{\lambda_l-l+1}-u^{-1}q^{-\lambda_l+l-1}}{q-q^{-1}}\mathcal{E}_{\Tc}^{\lambda} e_1\ot \cdots \ot e_m.
\eal
\eeq
Consider the action of $z_{k}(u), 1 \leq k <m$, by the recurrence relation of $\mathcal{E}_{\Tc}^{\lambda}$,
\beq
\bal
&z_{k}(u)\cdot \mathcal{E}_{\Tc}^{\lambda} e_1\ot \cdots \ot e_m\\
&=Y_m\cdots Y_{k+1} z_k(u)\cdot \mathcal{E}_{\Tc^-}^{\mu} e_1\ot \cdots \ot e_m\\
&=\prod_{l=1}^{k}\frac{uq^{\mu_l-l+1}-u^{-1}q^{-\mu_l+l-1}}{q-q^{-1}}
Y_m\cdots Y_{k+1} \mathcal{E}_{\Tc^-}^{\mu} e_1\ot \cdots \ot e_m\\
&=\prod_{l=1}^{k}\frac{uq^{\mu_l-l+1}-u^{-1}q^{-\mu_l+l-1}}{q-q^{-1}}
  \mathcal{E}_{\Tc}^{\lambda} e_1\ot \cdots \ot e_m,
\eal
\eeq
where $Y_m,\ldots,Y_{k+1}$ are some rational functions at $u=y_m,\ldots,u=y_{k+1}$ and $y_m,\ldots,y_{k+1}$ are the Jucys-Murphy elements of $\Hc_m$. And $\Tc^-$ is the standard tableau obtained from $\Tc$ by removing the boxes occupied by $k+1,\ldots,m$, denote by $\mu$ the shape of $\Tc^-$.

Hence by corresponding \eqref{pattern-tableau-corr},
\beq
\sqrt{c_{\lambda}} \mathcal{E}_{\Tc}^{\lambda}\cdot e_1\ot \cdots \ot e_m= \xi_{\Tc}\ot v_{\Tc} .
\eeq
\end{proof}

For any $I=(1^{\mu_1} ,\ldots   ,n^{\mu_m})=(i_1, \ldots, i_m)$, denote $\mu=(\mu_1,\ldots,\mu_m)$.
For any standard partition $\lambda\vdash m$, the set of all Young tableaux of shape $\lambda$ is denoted
as
$\mathrm{YT}(\lambda)$.
The set of semistandard Young tableaux of shape  $\lambda$ with entries at most $n$ is denoted by $\mathrm{SSYT}(\lambda, n).$
The set of standard Young tableaux of shape
$\lambda$ is denoted as  $\mathrm{SYT}(\lambda).$

We define a map
\beq
\bal
\theta_{\mu}: \mathrm{SYT}(\lambda) \rightarrow \mathrm{YT}(\lambda)\\
\eal
\eeq by replacing $r$ in $\Tc$  by $i_r$ for $1 \leq r \leq m$.
Then $\mathrm{SSYT}(\lambda, n)$ is contained in the image of $\theta_{\mu}$.
\begin{lemma}\label{schur-weyl-corr}
Let $I=(i_1 \leq \ldots \leq i_m)$ with $m_i(I)=\mu_i$, then
\beq
\mathcal{E}_{\Tc}^{\lambda}\cdot e_{i_1}\ot \cdots \ot e_{i_m}=\left\{\begin{array}{cc}
c\cdot \xi_{\theta_{\mu}(\Tc)}\ot v_{\Tc}&\text{ if } \theta_{\mu}(\Tc) \in \mathrm{SSYT}(\lambda, n),\\
0&\text{ if }\theta_{\mu}(\Tc) \notin \mathrm{SSYT}(\lambda, n),
\end{array}\right.
\eeq
where $c$ is some nonzero constant.
\end{lemma}
\begin{proof}
To start with, if $\theta_{\mu}(\Tc) \notin \mathrm{SSYT}(\lambda, n)$, then there exists $r,t$ are in the same column of $\Tc$ and $i_r=i_t$. We can assume $1 \leq r<t \leq m$ and $c_r(\Tc)-c_t(\Tc)=1$. So $i_r=i_{r+1}=\ldots=i_t$. Moreover, by the  recurrence relation and fusion procedure \eqref{fusion-procedure} of $\Ec_{\Tc}^{\lambda}$, we have that
\beq
\bal
&\Ec_{\Tc}^{\lambda}\cdot e_{i_1}\ot \cdots \ot e_{i_m}\\
&=Y_m\cdots Y_{t+1}\Ec_{\Tc^{-}}^{\mu} \cdot e_{i_1}\ot \cdots \ot e_{i_m}\\
&=Y_m\cdots Y_{t+1}\frac{1}{c_{\mu^T}}\phi_{\Tc^{-}}(z_1,\ldots,z_t)T_0^{-1}|_{z_1=1}\cdots |_{z_t=1}\cdot e_{i_1}\ot \cdots \ot e_{i_m},
\eal
\eeq
where  $Y_m,\ldots,Y_{t+1}$ are some rational functions at $u=y_m,\ldots,u=y_{t+1}$ and $y_m,\ldots,y_{t+1}$ are  Jucys–Murphy elements of $\Hc_m$. And $\Tc^-$ is the standard tableau obtained from $\Tc$ by removing the boxes occupied by $t+1,\ldots,m$, denote by $\mu$ the shape of $\Tc^-$.

From relation \eqref{T0-relation}, we can rewrite the fusion procedure as
\beq
\bal
\Ec_{\Tc^{-}}^{\mu} =&\frac{1}{c_{\mu^T}} T_0^{-1} \phi'_{\Tc^{-}}(z_1,\ldots,z_{t-1})\\
&\times T_1(q^{2c_1}z_1,q^{2c_{t}}z_t) \ldots T_{t-1}(q^{2c_{t-1}}z_{t-1},q^{2c_{t}}z_t)|_{z_1=1}\cdots |_{z_t=1}\\
=&\frac{1}{c_{\mu^T}} T_0^{-1} \phi'_{\Tc^{-}}(z_1,\ldots,z_{t-1})T_1(q^{2c_1}z_1,q^{2c_{t}}z_t) \ldots\\ &\cdot T_{r-1}(q^{2c_{r-1}}z_{r-1},q^{2c_{t}}z_t)|_{z_1=1}\cdots |_{z_t=1}\\
&\times T_{r}(q^{2c_{r}},q^{2c_{t}}) \ldots T_{t-1}(q^{2c_{t-1}},q^{2c_{t}}),
\eal
\eeq
where $\phi'_{\Tc^{-}}(z_1,\ldots,z_{t-1})$ is some $\Hc_{t-1}$-valued rational function  determined by $\phi_{\Tc^{-}}$. Note that
\ben
T_{r}(q^{2c_{r}},q^{2c_{t}})=T_r+\frac{q-q^{-1}}{q^{2(c_t-c_r)}-1}=T_r-q,
\een
and
\ben
\bal
&T_{r}(q^{2c_{r}},q^{2c_{t}}) \ldots T_{t-1}(q^{2c_{t-1}},q^{2c_{t}}) \cdot e_{i_1}\ot \cdots \ot e_{i_m}\\
&=const\cdot  (\check{R}_r-q)\cdot e_{i_1}\ot \cdots \ot e_{i_m}\\
&=0.
\eal
\een
Hence, $\Ec_{\Tc}^{\lambda}\cdot e_{i_1}\ot \cdots \ot e_{i_m}=0$ if $\theta_{\mu}(\Tc) \notin \mathrm{SSYT}(\lambda, n)$.

It follows form Lemma \ref{lemma1} that for $m\leq n$  and $\mu=(1^m)$,
\beq
\sqrt{c_{\lambda}} \mathcal{E}_{\Tc}^{\lambda}\cdot e_1\ot \cdots \ot e_m=\xi_{\Tc}\ot v_{\Tc} .
\eeq
(1) For $m\leq n$ and any $I=(i_1 \leq \ldots \leq i_m)$ with $m_i(I)=\mu_i$,
\beq
\bal
&\left(\overset{\rightarrow}{\prod}_{i_r>r}F_{i_r-1}\cdots F_r \right) \left(\overset{\leftarrow}{\prod}_{i_t<t}E_{i_t}\cdots E_{t-1}\right) e_1\ot \cdots \ot e_m\\
&=c\cdot e_{i_1}\ot \cdots \ot e_{i_m},
\eal
\eeq
where $c$ is some nonzero constant.

Moreover assume $t_0$ is the minimal such that $i_{t_0}<t_0$, then $i_{t_0-1}\geq t_0-1$. Since
$$
\begin{matrix}
\cdots & t_0-1 & < & t_0 & \cdots &  \\
 & \rotatebox{90}{$\geq$} & & \rotatebox{90}{$<$} &  & \\
\cdots  &i_{t_0-1} &\leq& i_{t_0}& \cdots &
\end{matrix}$$
so $i_{t_0-1}= i_{t_0}=t_0-1$.
According to coresponding \eqref{pattern-tableau-corr} and Gelfand-Tsetlin formula \cite{GT,UTS1,UTS2}, $E_{t_0-1}\xi_{\Tc}=0$ if $t_0-1,t_0$ are in the same column of $\Tc$. Otherwise,
\beq
\bal
E_{t_0-1}\xi_{\Tc}=c'\xi_{\theta_{\mu_0}(\Tc)},
\eal
\eeq
where $\mu_0=(1^{t_0-2},2,0,1^{m-t_0},0^{n-m})$ and $\theta_{\mu_0}(\Tc)$ is the unique semi-standard tableau of shape $\lambda$ and weight $\mu_0$ by replacing $t_0$ with $t_0-1$ in $\Tc$. Then
\ben
\bal
&\sqrt{c_{\lambda}} E_{t_0-1}\cdot \mathcal{E}_{\Tc}^{\lambda} e_1\ot \cdots \ot e_m\\
&=const\cdot \mathcal{E}_{\Lambda}^{\lambda} e_1\ot \cdots \ot e_{t_0-1}\ot e_{t_0-1} \ot e_{t_0+1}\ot \cdots \ot e_k\\
\eal
\een
On the other hand,
\ben
\bal
&E_{t_0-1}\cdot \xi_{\Tc}\ot v_{\Tc}=const\cdot \xi_{\theta_{\mu_0}(\Tc)}\ot v_{\Tc} .
\eal
\een
Here the constants are nonzero.
Therefore,
\ben
\bal
&\sqrt{c_{\lambda}} E_{t_0-1}\cdot \mathcal{E}_{\Tc}^{\lambda} e_1\ot \cdots \ot e_m=const\cdot \xi_{\theta_{\mu_0}(\Tc)}\ot v_{\Tc}
\eal
\een
for nonzero constant $c$.
Generally, if $i_{t}<t$, considering the order of actions $E_{i_t}\cdots E_{t-1}$ and these actions is not trivial, we have that the unique Gelfand-Tsetlin vector $\xi_{\Tc^{\lambda}_{\mu'}}$ and $\Tc^{\lambda}_{\mu'}$ is obtained from the semi-standard tableau in previous step by replacing  $t$ with $i_t$.

Finally, as the two actions of $\overset{\rightarrow}{\prod}_{i_r>r}F_{i_r-1}\cdots F_r $ and  $\overset{\leftarrow}{\prod}_{i_t<t}E_{i_t}\cdots E_{t-1}$ are commutative, we can show the similar results when $i_r>r$. Hence
\beq
\bal
&\left(\overset{\rightarrow}{\prod}_{i_r>r}F_{i_r-1}\cdots F_r \right) \left(\overset{\leftarrow}{\prod}_{i_t<t}E_{i_t}\cdots E_{t-1}\right)  \mathcal{E}_{\Tc}^{\lambda} e_1\ot \cdots \ot e_m\\
&=const\cdot \mathcal{E}_{\Lambda}^{\lambda} e_{i_1}\ot \cdots \ot e_{i_m}\\
\eal
\eeq
\beq
\bal
&\left(\overset{\rightarrow}{\prod}_{i_r>r}F_{i_r-1}\cdots F_r \right) \left(\overset{\leftarrow}{\prod}_{i_t<t}E_{i_t}\cdots E_{t-1}\right)  \mathcal{E}_{\Tc}^{\lambda}\cdot\xi_{\Tc}\ot v_{\Tc}\\
&=const\cdot \xi_{\theta_{\mu}(\Tc)}\ot v_{\Tc}.
\eal
\eeq
Both constants are nonzero.

(2) For $n<m$, as $U_q(\mathfrak{gl}_n)$ is a subalgebra of $U_q(\mathfrak{gl}_m)$, then the actions of $z_1(u),\ldots,z_n(u)$ unchanged in restricting  $U_q(\mathfrak{gl}_m)$ to $U_q(\mathfrak{gl}_n)$. Thus
\beq
\mathcal{E}_{\Tc}^{\lambda}\cdot e_{i_1}\ot \cdots \ot e_{i_m}=\left\{\begin{array}{cc}
c\cdot\xi_{\theta_{\mu}(\Tc)}\ot v_{\Tc} &\text{ if } \theta_{\mu}(\Tc) \in \mathrm{SSYT}(\lambda, n),\\
0&\text{ if } \theta_{\mu}(\Tc) \notin \mathrm{SSYT}(\lambda, n),
\end{array}\right.
\eeq
where $c$ is some nonzero constant.
\end{proof}

\begin{proof}[\textbf{Proof of theorem \ref{q-Kostant-thm}}]
First we note that if $\Tc$ is the unique pre-image of $\theta_{\mu}(\Tc)$, ( i.e. if $i_s=i_t$ in $\theta_{\mu}(\Tc)$, then $s,t$ are in the same row of $\Tc$), then
\beq
\bal
&\langle  \mathcal{E}_{\Tc}\cdot e_{i_1}\ot \cdots \ot e_{i_m}, \mathcal{E}_{\Tc'}\cdot e_{i_1}\ot \cdots \ot e_{i_m}\rangle\\
&=\langle  \mathcal{E}_{\Tc'}\mathcal{E}_{\Tc}\cdot e_{i_1}\ot \cdots \ot e_{i_m},e_{i_1}\ot \cdots \ot e_{i_m}\rangle\\
&=\delta_{\Tc\Tc'}\langle  \mathcal{E}_{\Tc}\cdot e_{i_1}\ot \cdots \ot e_{i_m}, e_{i_1}\ot \cdots \ot e_{i_m} \rangle \\
&=\delta_{\Tc\Tc'}\frac{1}{c_{\lambda}}
\langle  \sum_{\si \in \Sym_m}\langle T_{\si^{-1}}v_{\Tc},v_{\Tc}\rangle \check{R}_{\si}\cdot e_{i_1}\ot \cdots \ot e_{i_m}, e_{i_1}\ot \cdots \ot e_{i_m} \rangle \\
&=\delta_{\Tc\Tc'}\frac{1}{c_{\lambda}}
\langle  \sum_{\si \in \Sym_I}\langle T_{\si^{-1}}v_{\Tc},v_{\Tc}\rangle \check{R}_{\si}\cdot e_{i_1}\ot \cdots \ot e_{i_m}, e_{i_1}\ot \cdots \ot e_{i_m} \rangle.
\eal
\eeq
Using formula \eqref{Young-orthogonal-form}, we have that
\beq
\langle T_{\si^{-1}}v_{\Tc},v_{\Tc}\rangle=q^{l(\si)}.
\eeq
Hence,
\beq
\bal
\langle  \mathcal{E}_{\Tc}\cdot e_{i_1}\ot \cdots \ot e_{i_m}, \mathcal{E}_{\Tc'}\cdot e_{i_1}\ot \cdots \ot e_{i_m}\rangle=\delta_{\Tc\Tc'}\frac{m_{q^2}(I)}{c_{\lambda}}.
\eal
\eeq
By lemma \ref{schur-weyl-corr},
\beq
\sqrt{\frac{c_{\lambda}}{m_{q^2}(I)}}\mathcal{E}_{\Tc}\cdot e_{i_1}\ot \cdots \ot e_{i_m}= \xi_{\theta_{\mu}(\Tc)}\ot v_{\Tc}.
\eeq

Moreover for any standard Young tableaux of shape $\lambda$,
 if $\theta_{\mu}(\Tc_1)=\cdots=\theta_{\mu}(\Tc_s)$ ($s>1$), then
\begin{align*}
&\langle  \sqrt{\frac{c_{\lambda}}{m_{q^2}(I)}}\sum_{k=1}^s\mathcal{E}_{\Tc_k}\cdot e_{i_1}\ot \cdots \ot e_{i_m}, \sqrt{\frac{c_{\lambda}}{m_{q^2}(I)}}\sum_{k=1}^s\mathcal{E}_{\Tc_k}\cdot e_{i_1}\ot \cdots \ot e_{i_m}\rangle\\
&=\frac{c_{\lambda}}{m_{q^2}(I)}\langle  \sum_{k=1}^s\mathcal{E}_{\Tc_k}\cdot e_{i_1}\ot \cdots \ot e_{i_m},e_{i_1}\ot \cdots \ot e_{i_m}\rangle\\
&=\frac{1}{m_{q^2}(I)}\langle \sum_{k=1}^s\sum_{\si \in \Sym_m}\langle T_{\si^{-1}} v_{\Tc_k},v_{\Tc_k}\rangle \check{R}_{\si}\cdot e_{i_1}\ot \cdots \ot e_{i_m}, e_{i_1}\ot \cdots \ot e_{i_m} \rangle \\
&=\frac{1}{m_{q^2}(I)}\langle \sum_{k=1}^s\sum_{\si \in \Sym_I}\langle T_{\si^{-1}} v_{\Tc_k},v_{\Tc_k}\rangle \check{R}_{\si}\cdot e_{i_1}\ot \cdots \ot e_{i_m}, e_{i_1}\ot \cdots \ot e_{i_m} \rangle \\
&=\frac{1}{m_{q^2}(I)} \sum_{\si \in \Sym_I} \sum_{k=1}^s\langle T_{\si^{-1}} v_{\Tc_k},v_{\Tc_k}\rangle q^{l(\si)}.
\end{align*}
Since the subspace $V=\text{span}_{\mathbb{C}(q)}\{v_{\Tc_k}, 1 \leq k \leq s\}$ is a finite-dimensional module of $\Hc_{\mu}=\Hc_{\mu_1}\times \cdots \times \Hc_{\mu_m}$ and it is isomorphic to the tensor product of skew representation of $\Hc_{\mu_i}$, $ 1 \leq i \leq m$. As $\theta_{\mu}(\Tc_k)\neq 0$, the skew partition is disconnected and then it is a permutation module. Thus it contains the trivial representation $1_{\Hc_{\mu}}$  with multiplicity one. So
\beq
\bal
&\frac{1}{m_{q^2}(I)} \sum_{\si \in \Sym_I} \sum_{k=1}^s\langle T_{\si^{-1}} v_{\Tc_k},v_{\Tc_k}\rangle q^{l(\si)}\\
&=\langle \chi_q^V, 1_{\Hc_{\mu}}  \rangle\\
&=1.
\eal
\eeq
By Lemma \ref{schur-weyl-corr} we can know that
\beq
\sqrt{\frac{c_{\lambda}}{m_{q^2}(I)}}\mathcal{E}_{\Tc_k}\cdot e_{i_1}\ot \cdots \ot e_{i_m}=a_k \xi_{\theta_{\mu}(\Tc_1)}\ot v_{\Tc_k}.
\eeq
where $\sum_{k=1}^s |a_k|^2=1$.

So the right side of equation \eqref{general-Kostant-iden} equals to
\begin{align*}
&tr(\mathcal P_{\mu}\ot 1)\circ\Delta\circ \mathcal P_{\mu}|_{U^{\lambda}}\\
&=\langle  \sqrt{\frac{c_{\lambda}}{m_{q^2}(I)}}(\mathcal P_{\mu}\ot 1) \circ \Delta \circ\sum_{\Tc}\mathcal{E}_{\Tc}\cdot e_{i_1}\ot \cdots \ot e_{i_m}, \sqrt{\frac{c_{\lambda}}{m_{q^2}(I)}}\sum_{\Tc}\mathcal{E}_{\Tc}\cdot e_{i_1}\ot \cdots \ot e_{i_m}\rangle\\
&=\frac{c_{\lambda}}{m_{q^2}(I)}\langle  (\mathcal P_{\mu}\ot 1) \circ \Delta \circ\sum_{\Tc}\mathcal{E}_{\Tc}\cdot e_{i_1}\ot \cdots \ot e_{i_m}, e_{i_1}\ot \cdots \ot e_{i_m}\rangle.
\end{align*}

Since by the Equations \eqref{PR} and  \eqref{e:char}, we have that
\begin{align*}
&\frac{c_{\lambda}}{m_{q^2}(I)}  (\mathcal P_{\mu}\ot 1) \circ \Delta \circ\sum_{\Tc}\mathcal{E}_{\Tc}\cdot e_{i_1}\ot \cdots \ot e_{i_m}\\
&=\frac{1}{m_{q^2}(I)} (\mathcal P_{\mu}\ot 1)\circ\Delta\circ\sum_{\tau\in \Sym_I,\atop \mu \in \mathcal{M}(\Sym_m/\Sym_I)} \chi_q^{\lambda}(T_{\mu^{-1}\tau}) \check{R}_{\mu^{-1}} \check{R}_{\tau} \cdot e_{i_1}\ot \cdots \ot e_{i_m}\\
&=\frac{1}{m_{q^2}(I)} (\mathcal P_{\mu}\ot 1)\circ\Delta\sum_{\tau\in \Sym_I,\atop \mu \in \mathcal{M}(\Sym_m/\Sym_I)} \chi_q^{\lambda}(T_{\mu^{-1}\tau})q^{l(\tau)} e_{i_{\mu_1}}\ot \cdots \ot e_{i_{\mu_m}}\\
&=\frac{1}{m_{q^2}(I)}  \sum_{(j_1,\ldots,j_m)}\sum_{\tau\in \Sym_I,\atop \mu \in \mathcal{M}(\Sym_m/\Sym_I)} \chi_q^{\lambda}(T_{\mu^{-1}\tau})q^{l(\tau)}
(\mathcal P_{\mu}\ot 1) e_{j_1}\ot \cdots \ot e_{j_m}\ot x_{j_1,i_{\mu_1}}\cdots x_{j_m,i_{\mu_m}}  \\
&=\frac{1}{m_{q^2}(I)}  \sum_{\si \in \Sym_m}\sum_{\tau\in \Sym_I,\atop \mu \in \mathcal{M}(\Sym_m/\Sym_I)} \chi_q^{\lambda}(T_{\mu^{-1}\tau})q^{l(\tau)}
 e_{i_{\si_1}}\ot \cdots \ot e_{i_{\si_m}}\ot x_{i_{\si_1},i_{\mu_1}}\cdots x_{i_{\si_m},i_{\mu_m}},
\end{align*}
where the first sum in the fourth line runs over all sequence $(j_1\ldots,j_m)$ of $[n]$.

As the two actions of $U_q(\mathfrak{gl}_n)$ and  $\Hc_m$ are commutative, so
\beq
\bal
&tr(\mathcal P_{\mu}\ot 1)\circ\Delta\circ \mathcal P_{\mu}|_{U^{\lambda}}\\
&=\frac{c_{\lambda}}{m_{q^2}(I)}\langle  (\mathcal P_{\mu}\ot 1) \circ \Delta \circ\sum_{\Tc}\mathcal{E}_{\Tc}\cdot e_{i_1}\ot \cdots \ot e_{i_m}, e_{i_1}\ot \cdots \ot e_{i_m}\rangle\\
&=\frac{1}{m_{q^2}(I)} \sum_{\tau\in \Sym_I,\atop \mu \in \mathcal{M}(\Sym_m/\Sym_I)} \chi_q^{\lambda}(T_{\mu^{-1}\tau})q^{l(\tau)}
 x_{i_1,i_{\mu_1}}\cdots x_{i_{m},i_{\mu_m}}\\
&=\frac{1}{m_{q^2}(I)} \langle i_{1},\ldots ,i_{m}\mid \chi_q^{\lambda}X_1\cdots X_m \mid i_{1},\ldots,i_m\rangle\\
&=\frac{\Imm_{\chi_q^{\lambda}}(X_I)}{m_{q^2}(I)}.
\eal
\eeq
which completes the proof.
\end{proof}

\section{Quantum Littlewood correspondence I and II}
We have the quantum analog of the Littlewood correspondence I \cite{Li}[p. 118].
\begin{theorem}\label{LittlewoodI}
   Corresponding to any relation between Schur functions of total order $n$, we may replace each Schur function by the corresponding quantum immanant of complementary principal minors of the generator matrix $X$ of $A_q(\Mat_n)$ provided that every product is summed for all sequences of pairwise disjoint increasing ordered subsets of $[n]$.
\end{theorem}
\begin{proof}
It is sufficient to prove the theorem for the product of two Schur functions. Suppose that $|\mu|+|\nu|=n$  and $$\mathbb{S}_{\mu}\mathbb{S}_{\nu}=\sum_{\lambda \vdash n}c_{\mu\nu}^{\lambda}\mathbb{S}_{\lambda}.$$
It follows from Proposition \ref{induced-pro} that
\begin{align*}
\chi_q^{\mathrm{Ind}(V^{\mu}\ot V^{\nu})}=\sum_{\si \in \mathfrak{S}_n}T_{\si}\mathcal{E}^{\mu}\mathcal{E}^{\nu}T_{\si^{-1}}=\sum_{\lambda}c_{\mu\nu}^{\lambda}\chi_q^{\lambda}.
\end{align*}
Then we have that
\begin{equation}
\begin{aligned}
&\Imm_{\chi_q^{\mathrm{Ind}(V^{\mu}\ot V^{\nu})}}(X) \\
   &=\langle 1,\ldots ,n
\mid \chi_q^{\mathrm{Ind}(V^{\mu}\ot V^{\nu})} X_1\cdots X_{n} \mid 1,\ldots,n\rangle\\
   &=\sum_{\sigma\in
\mathfrak{S}_n}\langle 1,\ldots ,n \mid \check{R}_{\si}\mathcal{E}^{\mu}\mathcal{E}^{\nu} \check{R}_{\si^{-1}} X_1\cdots X_{n} \mid 1,\ldots,n\rangle\\
  \end{aligned}
\end{equation}
Let $\mathcal M(\mathfrak{S}_n/\Sym_{|\mu|}\times \Sym_{|\nu|})$ be the minimal length coset representative of $\mathfrak{S}_n/\Sym_{|\mu|}\times \Sym_{|\nu|}$. Then
\begin{equation}
\begin{aligned}
\mathcal M(\mathfrak{S}_n/\Sym_{|\mu|}\times \Sym_{|\nu|})=\{\rho \in \Sym_n | \rho_1 <\cdots<\rho_{|\mu|},\rho_{|\mu|+1} <\cdots<\rho_{n}\}.
  \end{aligned}
\end{equation}
Since any element $\sigma \in
\mathfrak{S}_n $ can be written as $ \rho \tau$, where $\rho \in \mathcal M(\mathfrak{S}_n/\Sym_{|\mu|}\times \Sym_{|\nu|})$, $\Imm_{\chi_q^{\mathrm{Ind}(V^{\mu}\ot V^{\nu})}}(X) $ is equal to
\begin{equation}
\begin{aligned}
&\sum_{ \tau \in \Sym_{|\mu|}
\times \Sym_{|\nu|},\atop \rho \in \mathcal M(\mathfrak{S}_n/\Sym_{|\mu|}\times \Sym_{|\nu|})}\langle  \rho_1,\ldots,\rho_n \mid \check{R}_{\tau}\mathcal{E}^{\mu}\mathcal{E}^{\nu} \check{R}_{\tau^{-1}} X_1\cdots X_n\mid \rho_1, \ldots, \rho_n \rangle.\\
   &=\sum_{(I_1,I_2)}{\Imm}_{\chi_q^{\mu}}(X_{I_1}) {\Imm}_{\chi_q^{\nu}}(X_{I_2}),
\end{aligned}
\end{equation}
the sum on $(I_1,I_2)$ is taken over all sequences $(I_1,I_2)$ of disjoint increasing ordered subsets of $[n]$ satisfying $|I_1|=|\mu|$ and $|I_2|=|\nu|$.

On the other hand,
\begin{equation}
\begin{aligned}
\Imm_{\chi_q^{\mathrm{Ind}(V^{\mu}\ot V^{\nu})}}(X)
   &=\langle 1,\ldots ,n
\mid \chi_q^{\mathrm{Ind}(V^{\mu}\ot V^{\nu})} X_1\cdots X_{n} \mid 1,\ldots,n\rangle\\
   &=\sum_{\lambda}c_{\mu\nu}^{\lambda}\langle 1,\ldots ,n \mid \chi_q^{\lambda} X_1\cdots X_{n} \mid 1,\ldots,n\rangle\\
   &=\sum_{\lambda}c_{\mu\nu}^{\lambda}\Imm_{\chi_q^{\lambda}}(X).
\end{aligned}
\end{equation}
Hence,
\[
\sum_{(I_1,I_2)}{\Imm}_{\chi_q^{\mu}}(X_{I_1}) {\Imm}_{\chi_q^{\nu}}(X_{I_2})=\sum_{\lambda}c_{\mu\nu}^{\lambda} \Imm_{\chi_q^{\lambda}}(X).
\]
\end{proof}

Let  $\psi_q^{\mu }$ (resp. $\phi_q^{\mu}$) be the induced characters of the sign character (resp. trivial character) of the parabolic subalgebra $\mathcal{H}_{\mu}$ of $\mathcal{H}_n$.
They can be decomposeed into  the irreducible $\mathcal{H}_n$-characters:
\begin{align*}
 \psi_q^{\mu}=\sum_{\lambda}K_{\lambda^{T},\mu}\chi_q^{\lambda}, \ \ \
\phi_q^{\mu}=\sum_{\lambda}K_{\lambda,\mu}\chi_q^{\lambda},
\end{align*}
where $K_{\lambda,\mu}$ are the Kostka numbers.

As a special case of Littlewood correspondence I, we obtain the identities in \cite{KS}[Theorem 5.4] which can be viewed as the quantum analog of Littlewood-Merris-Watkins identities \cite{Li,MW}.
\begin{corollary}\label{LMW-iden1}
  Fix a partition $\lambda=(\lambda_1,\dots,\lambda_l)$ of $n$, then
\begin{equation}
\begin{aligned}
   &\Imm_{\psi_q^{\lambda}}(X)=\sum_{(I_1,\dots,I_l)}{\det}_q(X_{I_1})\cdots {\det}_q(X_{I_l}),\\
   &\Imm_{\phi_q^{\lambda}}(X)=\sum_{(I_1,\dots,I_l)}\per_q(X_{I_1})\cdots \per_q(X_{I_l}),
\end{aligned}
\end{equation}
where the sums are over all sequences $(I_1,\dots,I_l)$ of pairwise disjoint increasing ordered subsets of $[n]$ satisfying $|I_j|=\lambda_j$.
\end{corollary}
\begin{proof}
By proposition \ref{induced-pro}, we can obtain
\begin{equation}
\begin{aligned}
  \psi_q^{\lambda}&=\sum_{\si \in \mathfrak{S}_n}T_{\si}\mathcal{E}^{(1^{\lambda_1})}\cdots \mathcal{E}^{(1^{\lambda_l})}T_{\si^{-1}},\\
  \phi_q^{\lambda}&=\sum_{\si \in \mathfrak{S}_n}T_{\si}\mathcal{E}^{(\lambda_1)}\cdots \mathcal{E}^{(\lambda_l)}T_{\si^{-1}}.
\end{aligned}
\end{equation}
By the isomorphism between the algebra of symmetric
functions and the Grothendieck ring of representations of Hecke algebras, it follows the identity in symmetric functions:
$$\sum_{\mu}K_{{\mu}^{T},\lambda}\mathbb{S}_{\mu}=\mathbb{S}_{(1^{\lambda_1})}\cdots \mathbb{S}_{(1^{\lambda_l})}.$$
By the quantum Littlewood correspondence I, we obtain the first identity.
The second identity can be proved by the same argument.
\end{proof}

Considering submatrix $X_I$ with multiset $I$, we obtain the Littlewood correspondence II for generic matrices \cite{Li}[p. 120].
\begin{theorem}\label{LittlewoodII}
   Corresponding to any relation between Schur functions, we may replace each Schur function by the corresponding (normalized) quantum immanant of a principal minor of generator matrices $X$ of $A_q(\Mat_n)$ provided that we sum for all principal minors of the appropriate order with non-decreasing ordered multisets of $[n]$.
\end{theorem}
\begin{proof}
  we only need to prove that the product of two Schur functions case, suppose that $|\mu|+|\nu|=n$  and $$\mathbb{S}_{\mu}\mathbb{S}_{\nu}=\sum_{\lambda \vdash n}c_{\mu\nu}^{\lambda}\mathbb{S}_{\lambda}.$$
Then from proposition \ref{induced-pro}, we have that
\begin{align*}
\chi_q^{\mathrm{Ind}(V^{\mu}\ot V^{\nu})}=\sum_{\si \in \mathfrak{S}_n}T_{\si}\mathcal{E}^{\mu}\mathcal{E}^{\nu}T_{\si^{-1}}=\sum_{\lambda}c_{\mu\nu}^{\lambda}\chi_q^{\lambda}.
\end{align*}
Hence,  for any non-decreasing ordered multisets $I=(1\leq i_1\leq \ldots \leq i_n \leq n)$ of $[n]$, we have that
\begin{equation}
\begin{aligned}
&\Imm_{\chi_q^{\mathrm{Ind}(V^{\mu}\ot V^{\nu})}}(X_I) \\
   &=\langle i_1,\ldots ,i_n
\mid \chi_q^{\mathrm{Ind} (V^{\mu}\ot V^{\nu})} X_1\cdots X_{n} \mid i_1,\ldots ,i_n\rangle\\
   &=\sum_{\sigma\in
\mathfrak{S}_n}\langle i_1,\ldots ,i_n \mid \check{R}_{\si}\mathcal{E}^{\mu}\mathcal{E}^{\nu} \check{R}_{\si^{-1}} X_1\cdots X_{n} \mid i_1,\ldots ,i_n\rangle\\
   &=\frac{m_{q^2}(I)}{m(I)}\sum_{\sigma\in\mathfrak{S}_n}\langle i_{\sigma_1},\dots ,i_{\sigma_n}\mid \mathcal{E}^{\mu} \mathcal{E}^{\nu}X_1\cdots X_n \mid i_{\sigma_1},\dots,i_{\sigma_n}\rangle.
\end{aligned}
\end{equation}

For each sequence $(I_1,I_2)$ of  non-decreasing ordered multisets of $[n]$ satisfying $|I_1|=|\mu|, \ |I_2|=|\nu|$ and the disjoint union of $I_1$ and $I_2$ is $I$,  there exists the minimal elements $\mu \in \mathcal M(\mathfrak{S}_n/\Sym_{|\mu|}\times \Sym_{|\nu|})$ such that $\mu^{-1} \cdot I= (I_1,I_2)$.
Hence, by Proposition  \ref{X_I}, we have that
\begin{equation*}
\begin{aligned}
&\Imm_{\chi_q^{\mathrm{Ind}(V^{\mu}\ot V^{\nu})}}(X_I) \\
&=\frac{m_{q^2}(I)}{m(I)}\sum_{ \tau \in \Sym_{|\mu|}\times \Sym_{|\nu|}, \atop \mu \in \mathcal M(\mathfrak{S}_n/\Sym_{|\mu|}\times \Sym_{|\nu|}) }\langle i_{\tau_{\mu_1}},\dots ,i_{\tau_{\mu_n}}\mid \mathcal{E}^{\mu} \mathcal{E}^{\nu} X_1 \cdots X_n \mid i_{\tau_{\mu_1}},\dots ,i_{\tau_{\mu_n}}\rangle\\
&=\frac{m_{q^2}(I)}{m(I)}\sum_{(I_1,I_2)}\frac{m(I)}{m_{q^2}(I_1)m_{q^2}(I_2)}{\Imm}_{\chi_q^{\mu}}(X_{I_1}) {\Imm}_{\chi_q^{\nu}}(X_{I_2})\\
&=\sum_{(I_1,I_2)}\frac{m_{q^2}(I)}{m_{q^2}(I_1)m_{q^2}(I_2)} \Imm_{\chi_q^{\mu}}(X_{I_1})
\Imm_{\chi_q^{\nu}}(X_{I_2}),
\end{aligned}
\end{equation*}
where the sums are over all sequences $(I_1,I_2)$ of  non-decreasing ordered multisets of $[n]$  and the disjoint union of $I_j, \ j=1,2$ is $I$.

On the other hand,
\begin{equation}
\begin{aligned}
\Imm_{\chi_q^{\mathrm{Ind}(V^{\mu}\ot V^{\nu})}}(X_I)
   &=\langle i_1,\ldots ,i_n
\mid \chi_q^{\mathrm{Ind}(V^{\mu}\ot V^{\nu})} X_1\cdots X_{n} \mid i_1,\ldots ,i_n \rangle\\
   &=\sum_{\lambda}c_{\mu\nu}^{\lambda}\langle i_1,\ldots ,i_n  \mid \chi_q^{\lambda} X_1\cdots X_{n} \mid i_1,\ldots ,i_n \rangle\\
   &=\sum_{\lambda}c_{\mu\nu}^{\lambda}\Imm_{\chi_q^{\lambda}}(X_I).
\end{aligned}
\end{equation}
Hence,
\[
\sum_{(I_1,I_2)}\frac{m_{q^2}(I)}{m_{q^2}(I_1)m_{q^2}(I_2)} \Imm_{\chi_q^{\mu}}(X_{I_1}) \Imm_{\chi_q^{\nu}}(X_{I_2}) =\sum_{\lambda}c_{\mu\nu}^{\lambda} \Imm_{\chi_q^{\lambda}}(X_I).
\]
\end{proof}

In particular, we have the quantum analogue of general Littlewood-Merris-Watkins identities \cite{KS,MW}.
\begin{corollary}
   Let $I=(1\leq i_1\leq \ldots \leq i_r \leq n)$ be an ordered multiset of $[n]$. Fix a partition $\lambda=(\lambda_1,\dots,\lambda_l)$ of $r$. Then
\begin{equation}\label{gene-qLMW}
\begin{aligned}
   &\Imm_{\psi_q^{\lambda}}(X_I)= \sum_{(I_1,\dots,I_l)}{\det}_q(X_{I_1})\cdots {\det}_q(X_{I_l}),\\
   &\Imm_{\phi_q^{\lambda}}(X_I)=\sum_{(I_1,\dots,I_l)}\frac{m_{q^2}(I)}{m_{q^2}(I_1)\cdots m_{q^2}(I_l)}\per_q(X_{I_1})\cdots \per_q(X_{I_l}),
\end{aligned}
\end{equation}
where the sums are taken over all sequences $(I_1,\dots,I_l)$ of non-decreasing ordered multisets of $[n]$ satisfying $|I_j|=\lambda_j$ and the disjoint union of $I_j, 1 \leq j \leq l$ is $I$.
\end{corollary}
\begin{proof}
Similar to the proof of corollary \ref{LMW-iden1}, we can prove these identities by the quantum Littlewood correspondence II.
\end{proof}

\section{$q$-Goulden-Jackson identities and quantum Littlewood correspondence III}
\subsection{The Bethe subalgebra of $A_q(\Mat_n)$}
Define $\alpha_0=\beta_0=1$, for $k<0$, $\alpha_k=\beta_k=0$ and for $k>0$,
\begin{equation}\label{genefcn1}
\begin{aligned}
  &\alpha_k=tr\Ec^{(1^{k})}X_1\cdots X_k=\sum_{I \subseteq [n] \atop |I|=k} {\det}_q(X_I), \\ &\beta_k=tr\Ec^{(k)}X_1\cdots X_k=\sum_{|I|=k}\frac{\per_q(X_I)}{m_{q^2}(I)},
\end{aligned}
\end{equation}
where  the ordered multisets $I$ with $k$ elements be non-decreasing.
Note that $\alpha_k=0$, if $k>n$. The Bethe subalgebra $\mathfrak{R}_n$ of $A_q(\Mat_n)$ is generated by $\alpha_k,\ 0 \leq k \leq n$.
Here we give a simple proof for the following result first obtained in \cite{DL}.

\begin{theorem}\label{t:commutativity-alpha}
$\{\alpha_k \mid k \in \mathbb{Z}\}$ in $A_q(\Mat_n)$ pairwise commute.
\end{theorem}
\begin{proof}
It is sufficient to prove that for all $r,s \in \mathbb{Z}$,
\beq\label{alpha-rela}
  \alpha_r \alpha_s=\alpha_s \alpha_r.
\eeq
Indeed, by the RTT relation \eqref{RT},
\ben
\bal
\alpha_r \alpha_s &=tr\Ec^{(1^{r})}\ot \Ec^{(1^{s})}X_1\cdots X_{r+s}\\
&=tr\Ec^{(1^{r})}\ot \Ec^{(1^{s})}\left(\prod^{\rightarrow}_{1 \leq i \leq r < j \leq r+s}R_{ij}^{-1} \right) X_{r+1}\cdots \\
& \quad \cdots X_{r+s}X_1\cdots X_{r}\left(\prod^{\leftarrow}_{1 \leq i \leq r < j \leq r+s}R_{ij}\right),
\eal
\een
where the idempotent $\Ec^{(1^{s})}$ is over the copies of $\mathrm{End}(\mathbb{C}^n)$ labeled by $\{r+1,\ldots,r+s\}$ and the ordered product be taken by $(i,j)$ precedes $(i',j')$ if $j<j'$, or if $j=j'$ but $i>i'$.

Hence by the property of trace,
we have that
\ben
\bal
\alpha_r \alpha_s
&=tr\left(\prod^{\leftarrow}_{1 \leq i \leq r < j \leq r+s}R_{ij}\right)\Ec^{(1^{r})}\ot \Ec^{(1^{s})}\left(\prod^{\rightarrow}_{1 \leq i \leq r < j \leq r+s}R_{ij}^{-1} \right) X_{r+1}\cdots \\
& \quad \cdots X_{r+s}X_1\cdots X_{r}.
\eal
\een
The equation \eqref{alpha-rela} can follow from the variant form of Yang-Baxter equation (YBE) \eqref{YBeq}:
\beq\label{var-YBE}
\bal
\check{R}_{ij}R_{ik}R_{jk}=R_{ik}R_{jk}\check{R}_{ij},\\
\check{R}_{jk}R_{ik}R_{ij}=R_{ik}R_{ij}\check{R}_{jk}.
\eal
\eeq
Indeed, since the idempotent can be written as
\[
\Ec^{(1^{r})}\ot \Ec^{(1^{s})}=\frac{1}{c_{(1^r)}c_{(1^s)}}\sum_{1\leq i<j\leq r \atop r+1 \leq k <l \leq r+s }\check{R}_{ij} \check{R}_{kl}.
\]
So by \eqref{var-YBE}, we have that
\ben
\bal
&\left(\prod^{\leftarrow}_{1 \leq i \leq r < j \leq r+s}R_{ij}\right)\Ec^{(1^{r})}\ot \Ec^{(1^{s})}\left(\prod^{\rightarrow}_{1 \leq i \leq r < j \leq r+s}R_{ij}^{-1} \right)\\
&=\Ec^{(1^{r})}\ot \Ec^{(1^{s})}\left(\prod^{\leftarrow}_{1 \leq i \leq r < j \leq r+s}R_{ij}\right)\left(\prod^{\rightarrow}_{1 \leq i \leq r < j \leq r+s}R_{ij}^{-1} \right)\\
&=\Ec^{(1^{r})}\ot \Ec^{(1^{s})}.
\eal
\een
Hence,
\[
\alpha_r \alpha_s =tr\Ec^{(1^{r})}\ot \Ec^{(1^{s})} X_{r+1}\cdots  X_{r+s}X_1\cdots X_{r}=\alpha_s\alpha_r.
\]
\end{proof}

We set
\begin{align}
\lambda(t)=\sum_{k=0}^{\infty}t^k \alpha_k,\qquad
\si(t)=\sum_{k=0}^{\infty}t^k \beta_k.
\end{align}

By the  Littlewood-Richardson rule \cite{Sa}, we have a useful lemma in the following.
\begin{lemma}
For any positive integer $1 \leq z \leq n-1$,
\begin{equation}\label{ind-decom}
  \mathrm{Ind}_{\mathcal{H}_{n-z} \times \mathcal{H}_{z}}^{\mathcal{H}_n}(\chi_q^{(n-z)}\ot \chi_q^{(1^{z})})=\chi_q^{(n-z+1,1^{z-1})}+\chi_q^{(n-z,1^z)}.
\end{equation}
\end{lemma}

We have the generalized MacMahon Master Theorem \cite{GLZ,JLZ,Mac}:
\begin{theorem}\label{Mac}
$\lambda(-t)\times \si(t) =1.$
\end{theorem}
\begin{proof}
It is sufficient to show that
\begin{equation}\label{sum-term}
\sum_{r=0}^k (-1)^r tr\Ec^{(r)}\ot \Ec^{(1^{k-r})}X_1\cdots X_k=0,
\end{equation}
where the idempotent $\Ec^{(1^{k-r})}$ is over the copies of $\mathrm{End}(\mathbb{C}^n)$ labeled by $\{r+1,\ldots,k\}$.
By the proof of Proposition \ref{imm-char}, we have that
\ben
\bal
&tr\Ec^{(r)}\ot \Ec^{(1^{k-r})}X_1\cdots X_k\\
=&\sum_{I}\frac{1}{m(I)}\sum_{\si \in \mathfrak{S}_k}\langle i_{\sigma_1},\dots ,i_{\sigma_k}\mid \Ec^{(r)}\ot \Ec^{(1^{k-r})}X_1\cdots X_k \mid i_{\sigma_1},\dots ,i_{\sigma_k}\rangle\\
=&\sum_{I}\frac{1}{m_{q^2}(I)}\sum_{\si \in \mathfrak{S}_k}\langle i_1,\dots ,i_k\mid \check{R}_{\si}\Ec^{(r)}\ot \Ec^{(1^{k-r})}\check{R}_{\si^{-1}}X_1\cdots X_k \mid i_1,\dots ,i_k\rangle.
\eal
\een
Since proposition \ref{induced-pro} and equation \eqref{ind-decom}, for $1 \leq r \leq k-1$,
\[
\check{R}_{\si}\mathcal{E}^{(r)}\mathcal{E}^{(1^{\{r+1,\ldots,k\}})}\check{R}_{\si^{-1}}=\chi_q^{(r+1,1^{k-r-1})}+\chi_q^{(r,1^{k-r})}.
\]
Therefore the telescoping sum \eqref{sum-term} equals zero.
\end{proof}
In particular, all the elements $\alpha_k, \beta_k$ commute with each other.
Consider the $q$-permutation operator $P^q \in \mathrm{End}(\mathbb{C}^n\ot \mathbb{C}^n)$ defined by
\[
P^q=\sum_{i}e_{ii}\ot e_{ii}+q\sum_{i>j}e_{ij}\ot e_{ji}+q^{-1}\sum_{i<j}e_{ij}\ot e_{ji}.
\]

Let $Y, Z$ be $n\times n$ matrices over  $A_{q}(\Mat_{n})$. Denote $Y*Z=tr_1 P^q Y_1Z_2$.
Let $X^{[k]}$ be the $k$th power of $X$ under the multiplication $*$, i.e.
\begin{align}\label{e:power-m}
X^{[0]}=1,\   X^{[1]}=X, \ X^{[k]}=X^{[k-1]}*X,\ k>1.
\end{align}
We denote
\begin{align}
\psi(t)=\sum_{k=0}^{\infty}t^k\gamma_{k+1},
\end{align}
where $\gamma_k=trX^{[k]}$ can be viewed as $q$-analog of power sums polynomials.

The following Newton identities follow from the quantum
MacMahon Master Theorem \cite{JLZ}.
\begin{theorem}[Newton's identities]\label{Newton-iden}
Let $X$ be the generator matrix of $A_q(\Mat(n))$.  Then
\begin{align}
&\partial_t \lambda(-t)=-\lambda(-t) \psi(t),\\
&\partial_t \si(t)= \psi(t)\si(t).
 \end{align}
\end{theorem}

The $q$-characteristic polynomial of generator matrix $X$ of $A_q(\Mat(n))$ is defined as
\begin{equation}
  \text{char}_q(X,t)=\sum_{k=0}^{n}(-1)^k \alpha_{k}t^{n-k}.
\end{equation}
We have the quantum Cayley-Hamilton Theorem in \cite{ZJJ,JLZ}.
\begin{theorem}
Let $X$ be the generator matrix of $A_q(\Mat(n))$. Then
\begin{equation}
\sum_{k=0}^n(-1)^k \alpha_k X^{[n-k]}=0.
\end{equation}
\end{theorem}

\subsection{$q$-Goulden-Jackson identities and Quantum Littlewood correspondence III }

Fix a partition $\lambda \vdash r$. On the Bethe algebra $\mathfrak{R}_n$ generated by $\alpha_i$ or $\beta_i$ \eqref{genefcn1}, we denote
\[
A=(\alpha_{\lambda_i^{T}-i+j})_{\lambda_1\times \lambda_1},\ \ \ \ \  B=(\beta_{\lambda_i-i+j})_{\lambda_1^T\times \lambda_1^T}.
\]
Then we have the following two special elements in $\mathfrak{R}_n$:
\begin{align*}
  &\det(A)=\sum_{\mu}K_{\lambda^{T},\mu}^{-1}\alpha_{\mu_1}\cdots \alpha_{\mu_{\lambda_1}}, \\
  &\det(B)=\sum_{\mu}K_{\lambda,\mu}^{-1}\beta_{\mu_1}\cdots \beta_{\mu_{\lambda_1}}.
\end{align*}

The following is the quantization of the  Goulden-Jackson identities that appeared as Theorem 2.1 in \cite{GJ} and
Theorem 3.2 in \cite{KS}.  It can be viewed as the generalization of the Jacobi-Trudi identity of Schur polynomials.
\begin{theorem}\label{quantum-JT} We have that
\begin{equation}
\det(A)=\det(B)=tr(\mathcal{E}_{\Tc}^{\lambda}X_1\cdots X_r) =\sum_{I}\frac{\Imm_{\chi_q^{\lambda}}(X_I)}{m_{q^2}(I)},
\end{equation}
where the sum is over all non-decreasing ordered multisets $I$ of $[n]$ satisfying $|I|=r$.
\end{theorem}
\begin{proof}
  By definition,
  \begin{equation*}
  \begin{aligned}
     &\det(A)\\
     &=\sum_{\mu}K_{\lambda^{T},\mu}^{-1}tr\mathcal{E}^{(1^{\mu_1})}\cdots \mathcal{E}^{(1^{\mu_{\lambda_1}})}X_1\cdots X_r\\
     &=\sum_{\mu}\sum_{J}\langle j_1,\dots ,j_r \mid K_{\lambda^{T},\mu}^{-1}\mathcal{E}^{(1^{\mu_1})}\cdots \mathcal{E}^{(1^{\mu_{\lambda_1}})}X_1\cdots X_r\mid  j_1,\dots ,j_r  \rangle,
  \end{aligned}
  \end{equation*}
and
\ben
\bal
&tr(\mathcal{E}_{\Tc}^{\lambda}X_1\cdots X_r)\\
 &= \sum_{J} \langle  j_1,\dots ,j_r  \mid \mathcal{E}_{\Tc}^{\lambda}X_1\cdots X_r \mid  j_1,\dots ,j_r  \rangle\\
&= \sum_{I}\sum_{\si \in \Sym_r} \frac{1}{m(I)}\langle  i_{\si_1},\dots ,i_{\si_r} \mid \mathcal{E}_{\Tc}^{\lambda}X_1\cdots X_r \mid  i_{\si_1},\dots ,i_{\si_r} \rangle\\
&=\sum_{I}\frac{\Imm_{\chi_q^{\lambda}}(X_I)}{m_{q^2}(I)},
\eal
\een
  where the sum in the first line is over all sequences $J$ with $r$ elements and the sum in the second line runs over all non-decreasing multisets $I$ of $[n]$ satisfying $|I|=r$.

  By the  Proposition  \ref{X_I}, we have that
   \begin{equation*}
  \begin{aligned}
     \det(A)&=\sum_{\mu}\sum_{J}\langle  j_1,\dots ,j_r  \mid K_{\lambda^{T},\mu}^{-1}\mathcal{E}^{(1^{\mu_1})}\cdots \mathcal{E}^{(1^{\mu_{\lambda_1}})}X_1\cdots X_r\mid  j_1,\dots ,j_r  \rangle\\
     &=\sum_{\mu}\sum_{J'}\frac{1}{m_{q^2}(J')}\langle  j'_1,\dots ,j'_r  \mid K_{\lambda^{T},\mu}^{-1} \sum_{\si\in \mathfrak{S}_r}\check{R}_{\si}\mathcal{E}^{(1^{\mu_1})}\cdots \\
     &\quad \cdots \mathcal{E}^{(1^{\mu_{\lambda_1}})}\check{R}_{\si^{-1}} X_1\cdots X_r\mid j'_1,\dots ,j'_r \rangle\\
     &=\sum_{\mu}\sum_{J'}\frac{1}{m_{q^2}(J')}\langle j'_1,\dots ,j'_r \mid K_{\lambda^{T},\mu}^{-1} \psi_q^{\mu} X_1\cdots X_r\mid  j'_1,\dots ,j'_r \rangle\\
     &=\sum_{J'}\frac{1}{m_{q^2}(J')}\langle  j'_1,\dots ,j'_r \mid\chi_q^{\lambda} X_1\cdots X_r\mid  j'_1,\dots ,j'_r\rangle\\
     &=\sum_{J'}\frac{\Imm_{\chi_q^{\lambda}}(X_{J'})}{m_{q^2}(J')},
  \end{aligned}
  \end{equation*}
  where the second sum of the second line runs over all ordered multisets $J'=(j'_1\leq \dots \leq j'_r)$ of $[n]$, and the fourth line is because
   $\sum_{\mu}K_{\lambda^{T},\mu}^{-1}\psi_q^{\mu}=\chi_q^{\lambda}$. We can follow $\det(B)=tr\mathcal{E}^{\lambda} X_1\cdots X_r$ by the same method.
\end{proof}


The quantum immanant can be also expressed as a Schur polynomial at particular arguments. More accurately, we have the quantum analog of Littlewood correspondence III \cite{Li}[p. 121].
\begin{theorem}\label{t:Litt3}
   Let $\lambda \vdash r\leq n$ and $\omega_1,\dots,\omega_n$ be the solutions of $\text{char}_q(X,t)=0$ over the algebraic closure field of fraction $\mathbb{C}(q, \alpha_1,\ldots,\alpha_n)$, then
\begin{equation}\label{Littlewood3}
\mathbb{S}_{\lambda}(\omega_1,\dots,\omega_n)=\sum_{I}\frac{\Imm_{\chi_q^{\lambda}}(X_I)}{m_{q^2}(I)},
\end{equation}
where the sum is over all non-decreasing ordered multisets $I$ of $[n]$  satisfying $|I|=r$.
In particular,
\beq
\bal
e_k(\omega_1,\dots,\omega_n)&=\sum_{I}{\det}_q(X_{I}),\\
h_k(\omega_1,\dots,\omega_n)&=\sum_{I}\frac{\per(X_I)}{m_{q^2}(I)},\\
p_k(\omega_1,\dots,\omega_n)&=trX^{[k]}.
\eal
\eeq
\end{theorem}
\begin{proof}
By  Vieta's formulas,  for $1 \leq k \leq n$, we have that $\alpha_k=e_k(\omega_1,\dots,\omega_n)$, where $e_k$ are the elementary symmetric functions. So
$$e_k(\omega_1,\dots,\omega_n)=\sum_{I \subseteq [n]}{\det}_q(X_{I}),$$
where the sum runs over all increasing ordered multisets $I$ with $k$ elements.
Hence, we can reduce \eqref{Littlewood3} by the $q$-Goulden-Jackson identities  and the dual Jacobi-Trudi identity:
$$\mathbb{S}_{\lambda}(\omega_1,\dots,\omega_n)=\sum_{\mu}K_{\lambda^{T},\mu}^{-1}e_{\mu_1}\cdots e_{\mu_{\lambda_1}}=\sum_{I}\frac{\Imm_{\chi_q^{\lambda}}(X_I)}{m_{q^2}(I)}.$$
Finally, by the Newton's identities in theorem \ref{Newton-iden}, $\gamma_k=trX^{[k]}$ can be expressed as the same linear combination of $e_{\mu}(\omega_1,\dots,\omega_n), \mu \vdash k$ as $p_k(\omega_1,\dots,\omega_n)$.
\end{proof}

Note that $\alpha_k=0$, if $k>n$. The Bethe subalgebra $\mathfrak{R}_n$ of $A_q(\Mat_n)$ generated by $\alpha_k,\ 0 \leq k \leq n$ is a commutative algebra.  And $\mathfrak{R}_n$ is the algebra of coinvariants  for the conjugation coaction of  $\GL_q(n)$ 
on $A_q(\Mat_n)$, see \cite{DL}.

Consider the algebra homomorphism $\Phi: A_q(\Mat_n) \rightarrow \mathbb{C}[x_1,\ldots,x_n]$ defined as
\[
x_{ij}\mapsto 0, \quad i\neq j, \qquad
x_{ii}\mapsto x_i, \quad i=1,\ldots,n.
\]
Denote $\textbf{Sym}^{(n)}=\mathbb{C}[x_1,\ldots,x_n]^{\mathfrak{S}_n}$ the algebra of symmetric polynomials in $n$ variables.
\begin{theorem}\label{iso-sym} The restriction of $\Phi$:
  $$\Phi |_{\mathfrak{R}_n} :\mathfrak{R}_n\rightarrow \textbf{Sym}^{(n)}$$
  is an algebra isomorphism. And $\{ \sum_{I} \frac{\Imm_{\chi_q^{\lambda}}(X_I)}{m_{q^2}(I)}\mid \lambda\vdash n\}$ is a basis of $\mathfrak{R}_n$.
\end{theorem}

\begin{proof}
First it is direct that $\Phi(\alpha_k)=e_k(x_1,\ldots,x_n)$. So the map $\Phi |_{\mathfrak{R}_n}$ is an epimorphism by the fundamental theorem of symmetric polynomials.
By  the quantum Littlewood correspondence III and $\Phi|_{\mathfrak{R}_n}$ is an isomorphism, $\{ \sum_{I}\frac{\Imm_{\chi_q^{\lambda}}(X_I)}{m_{q^2}(I)}\}$ is a basis of $\mathfrak{R}_n$.
\end{proof}

The following theorem can be viewed as a generalization of the relations between  Schur functions and power sums functions in \cite{Li}.
\begin{theorem}\label{rela-SchurPower}
Let $\lambda \vdash r\leq n$ and $\omega_1,\dots,\omega_n$ be the solutions of $\text{char}_q(X,t)=0$ over the algebraic closure field of fraction $\mathbb{C}(q, \alpha_1,\ldots,\alpha_n)$, then
  $$\sum_{I}\frac{\Imm_{\chi_q^{\lambda}}(X_I)}{m_{q^2}(I)}=\frac{\Imm_{\chi^{\lambda}}(\Gamma_r)}{r!},$$
where the lower Hessenberg matrix  $\Gamma_r$ is
$$
\Gamma_r=\begin{pmatrix}
\gamma_1 & -1 &  & &  \\
\gamma_2 & \gamma_1 & -2 &  & \\
\gamma_3 & \gamma_2 & \gamma_1 & -3 & \\
\vdots & \vdots & \ddots & \ddots &  \ddots\\
\gamma_r & \gamma_{r-1} & \cdots & \cdots & \gamma_1
\end{pmatrix}.$$
\end{theorem}
\begin{proof}
By the Newton identities in Theorem \ref{Newton-iden} and Cramer's rule, we have that
\[
\alpha_{r}=\frac{\det(\Gamma_r)}{r!}.
\]
Moreover, the Schur function in \cite{Li} can be defined by
\[
r!\ \mathbb{S}_{\lambda}=\Imm_{\chi^{\lambda}}(\mathfrak{P}_r),
\]
where the matrix $\mathfrak{P}_r$ is  obtained by replacing  $\gamma_i$ with the power sums polynomial $p_i$ in matrix $\Gamma_r$. According to the isomorphism $\Phi|_{\mathfrak{R}_n}$ and  the quantum Goulden-Jackson identities, we have that
\ben
\bal
\sum_{I}\frac{\Imm_{\chi_q^{\lambda}}(X_I)}{m_{q^2}(I)}&=
\sum_{\mu}K_{\lambda^{T},\mu}^{-1}\alpha_{\mu_1}\cdots \alpha_{\mu_{\lambda_1}}
=\frac{\Imm_{\chi^{\lambda}}(\Gamma_r)}{r!}.
\eal
\een
\end{proof}

\bigskip
\centerline{\bf Acknowledgments}
\medskip
The work is supported in part by the National Natural Science Foundation of China (grant nos.
12001218 and 12171303), the Simons Foundation (grant no. MP-TSM-00002518),
and the Fundamental Research Funds for
 the Central Universities (grant nos. CCNU22QN002 and
CCNU24JC001).

\bibliographystyle{amsalpha}

\end{document}